\documentclass[12pt]{amsart}
\usepackage{amscd,latexsym, amsthm,amsfonts,amssymb,amsmath,amsxtra}
\usepackage{mathptmx}
\DeclareMathAlphabet{\mathcal}{OMS}{cmsy}{m}{n}
\usepackage[all]{xy}
\usepackage{setspace}
\usepackage{cases}
\usepackage{graphicx,epsfig}
\usepackage{extarrows}
\usepackage{mathrsfs}
\usepackage{enumerate}
\usepackage{hyperref}   
\usepackage[square,sort&compress,comma,numbers]{natbib}
\usepackage{chemarrow}
\usepackage{mathtools}
\usepackage{tikz-cd}
\usetikzlibrary{decorations.pathmorphing}
\usepackage[normalem]{ulem}
\usepackage{units} 

\usepackage[top=25mm, outer=20mm, bottom=25mm, inner=25mm]{geometry}

\newcommand{\cancong}{\ensuremath{{\ \displaystyle \mathop{\cong}^{\mathrm{can}}}\ }}
\newcommand{\comment}[1]{} 

\newtheorem{innercustomthm}{Theorem}
\newenvironment{customthm}[1]{\renewcommand\theinnercustomthm{#1}\innercustomthm}{\endinnercustomthm}
\theoremstyle{remark}
\newtheorem{theorem}{\rm{\textbf{Theorem}}}[section]
\newtheorem{corollary}[theorem]{\rm{\textbf{\textbf{Corollary}}}}
\newtheorem{lemma}[theorem]{\rm{\textbf{Lemma}}}
\newtheorem{proposition}[theorem]{\rm{\textbf{\textbf{Proposition}}}}
\newtheorem{conjecture}[theorem]{\rm{\textbf{Conjecture}}}
\newtheorem{definition}[theorem]{\rm{\textbf{Definition}}}
\newtheorem{example}[theorem]{\rm{\textbf{Example}}}
\newtheorem{remark}[theorem]{\rm{\textbf{Remark}}}

\author[Q.~Yan]{Qijun Yan}
\title[A relation between zip stacks and moduli stacks of truncated local shtukas]{A relation between zip stacks and moduli stacks of truncated local shtukas}
\address{Beijing Institute of Mathematical Sciences and Applications (BIMSA), Beijing, 101\,~408,  China}
\email{yanqmath\symbol{64}bimsa.cn}
\date{}
\begin{document}

\begin{abstract}
Let \(G\) be a reductive group over a field \(k\), and let \(\mu\) be a cocharacter of \(G\). 
We prove that Viehmann's double coset spaces associated with \((G, \mu)\) are representable 
by certain Lusztig varieties, and establish a similar result for the mixed characteristic case. 
This representability enables a comparison between the moduli stacks of truncated local shtukas 
and zip stacks. Over a perfect field of positive characteristic, we establish a homeomorphism 
between the coarse moduli stack of \(1\text{-}1\)-truncated local \(G\)-shtukas and that of \(G\)-zips, thereby enriching our understanding of zip period maps in the context of Shimura varieties.
\end{abstract}
	\maketitle
	
	\section{Introduction}
 
	\subsection{Background and motivation}\label{Motivation}
	
	Let $p$ be a prime number and consider a reductive group $G$ defined over $ \mathbb{F}_p$. Denote by $\mathcal{L} G$ and $\mathcal{K} := \mathcal{L}^{+} G$ the loop group and positive loop group of $G$, respectively. Let $\mathcal{K}_1$ be the kernel of the reduction map $ \mathcal{K} \to G$; see \S \ref{S:loopgp} for definitions. Let $\mu: \mathbb{G}_{m,k} \to T$ be a dominant cocharacter  of $G$. Denote by $\mathcal{C}^{\mu} = \mathcal{K} \mu(t) \mathcal{K} \subseteq \mathcal{L}G$ the Cartan cell determined by $\mu$.

  \begin{definition}
We call the following fpqc quotient sheaf \emph{Viehmann's double coset space of type~$\mu$}:
\[
\prescript{}{1}{\mathcal{C}_1^\mu} := \mathcal{K}_1 \backslash \mathcal{C}^{\mu} / \mathcal{K}_1 
\subseteq \mathcal{K}_1 \backslash \mathcal{L}G / \mathcal{K}_1.
\]
\end{definition}

    Let $ k $ be an algebraically closed field of characteristic $ p $, and let $\sigma: k \to k$ be the absolute Frobenius of $ k $. Also denote by $\sigma$ the induced endomorphism on $ k[[t]] $. This Frobenius induces an endomorphism on $\mathcal{K}$ and on $\mathcal{L} G$, again denoted by $\sigma$. Denote by $\prescript{}{1}{\mathcal{C}_1^{\mu}}(k)/\mathrm{Ad}_{\sigma} \mathcal{K}(k)$ the orbit space of $\mathcal{K}(k)$-$\sigma$ conjugacy classes. Viehmann studied the $ \mathcal{K}(k)$-$\sigma$ conjugacy classes in the big coset space $  \mathcal{K}_1(k) \backslash \mathcal{L}G(k) / \mathcal{K}_1(k) $ in \cite{ViehmannTruncation1}. One of her results can be rephrased as: there exists an explicitly defined homeomorphism between the finite topological spaces
    \begin{equation}\label{Eq:homeo1}
		\prescript{J}{}{W} \ \overset{\epsilon_1}{\cong}\ \prescript{}{1}{\mathcal{C}_1^{\sigma(\mu)}}(k)/\mathrm{Ad}_{\sigma} \mathcal{K}(k),
	\end{equation}
   where $\prescript{J}{}{W}$ denotes the minimal length representatives of the coset space $W/W_{M_\mu}$, with $W$ being the Weyl group of $G$, and $J$ (a subset of the simple reflections of $W$) being the type of the centralizer $M_\mu = \mathrm{Cent}_{G}(\mu)$ of $\mu$. The set $\prescript{J}{}{W}$ is equipped with a partial order topology defined in \cite[Def.~6.1]{PinkWedhornZiegler1}, generalizing that in \cite{HeGstabwond}. The topology on the right-hand side of \eqref{Eq:homeo1} is determined by orbit-closure relations.

	On the other hand, the theory of $G$-zips ~\cite{Moonen&WedhornDiscreteinvariants, PinkWedhornZiegler1, PinkWedhornZiegler2} associated with the datum $(G, \mu)$ yields a smooth Artin stack of dimension $0$, denoted ${G\text{-}\mathsf{Zip}}^{\mu}$. It is known that ${G\text{-}\mathsf{Zip}}^{\mu}$ is a quotient stack $[G/\mathsf{E}_{\mu}^{\mathrm{PWZ}}]$ of $G$ by some zip group $\mathsf{E}_{\mu}^{\mathrm{PWZ}} \subseteq P_-^\sigma \times P_+$~\cite{PinkWedhornZiegler2}, and that its topology is described by the homeomorphism
	\begin{equation}\label{Eq:homeo2}
		\prescript{J}{}{W} \overset{\epsilon_2}{\cong} G(k)/\mathsf{E}_{\mu}^{\mathrm{PWZ}}(k) \cong |{G\text{-}\mathsf{Zip}}^{\mu}|. 
	\end{equation}
	Establishing either $\epsilon_1$ or $\epsilon_2$ is nontrivial; however, no direct link is known between the right-hand sides of~\eqref{Eq:homeo1} and \eqref{Eq:homeo2}, which might be considered satisfactory from a geometric perspective. The desire to establish such a link is understandable: while the right-hand side of \eqref{Eq:homeo1} parameterizes 1-truncated local $G$-shtukas of type $\mu$ over $k$, the right-hand side of \eqref{Eq:homeo2} parameterizes 1-truncated $p$-divisible groups of type $\mu$ over $k$, reflecting a \emph{number field side versus function field side} contrast. A key step toward establishing this link is to geometrize the right-hand side of~\eqref{Eq:homeo1}, based on Theorem~\ref{MainThm:rep} below, which underpins this work.

	Recall that the zip stack ${G\text{-}\mathsf{Zip}}^{\mu}$ serves as the period domain for zip period maps of Shimura varieties of type $\mu$. For historical context, the reader is referred to the introduction of \cite{Yan24}. Additionally, in \cite{Yan18}, we constructed a variant of the zip period map, targeting the fpqc sheaf 
	\[
	\prescript{}{1}{\mathcal{C}_1^{\mu'}}/\mathrm{Ad}_{\sigma} \mathcal{K}=\prescript{}{1}{\mathcal{C}_1^{\mu'}}/\mathrm{Ad}_{\sigma} G
	\]
	where $\mu' = p \sigma(\mu)$ (the multiplication between cocharacters is written additively). There, we establish a \emph{bijection} between the right-hand sides of \eqref{Eq:homeo1} and \eqref{Eq:homeo2}, compatible with the maps in these equations. However, one cannot directly declare our bijection as a homeomorphism without relying on established homeomorphism facts for both \eqref{Eq:homeo1} and \eqref{Eq:homeo2} (cf. \cite{WedhornViehmann2013}). 
	The quotient sheaves mentioned are too coarse to be regarded as ambient geometrizations. The usual belief was that, since the action of \(\mathcal{K}_1 \times \mathcal{K}_1\) on \(\mathcal{C}^{\mu'}\) is not free, one should not expect \(\prescript{}{1}{\mathcal{C}_1^{\mu'}}\) to be representable. Ironically, we discovered Theorem~\ref{MainThm:rep} while seeking a more refined argument for this “negative expectation”, intended to explain why the variant of the zip period map in~\cite{Yan18} does not admit an algebraic stack as its target. 
	
	\subsection{Main results}
	Let $k$ be an \emph{arbitrary field}. Associated with the cocharacter $\mu$ are a pair of parabolic subgroups $P_\pm = P_{\mu^{\pm 1}}$ of $G$. Let $U_{\pm} \subseteq P_{\pm}$ be their unipotent radicals, and let $M = M_\mu$ be their common Levi subgroup. 
	
	\begin{customthm}{A}[\protect{Representability of $\prescript{}{1}{\mathcal{C}_1^\mu}$}]\label{MainThm:rep}
		The fpqc sheaf $\prescript{}{1}{\mathcal{C}_1^\mu} = \mathcal{K}_1 \backslash \mathcal{K}\mu(t) \mathcal{K}/\mathcal{K}_1$ is represented by a geometrically connected, smooth, separated scheme of dimension $\dim G$ and of finite type over $k$. It is strongly quasi-affine and a $G^2 := G \times_k G$-spherical variety. Specifically, we have an isomorphism of fqpc sheaves
		\begin{equation}\label{Eq:KeyIsom}
			\mathsf{E}_{\mu} \backslash G^2 \cong \prescript{}{1}{\mathcal{C}_1^\mu}, \quad \ (g, h) \mapsto  \  g^{-1}\mu(t)h.
		\end{equation}
		where $\mathsf{E}_{\mu} \subseteq P_- \times P_+$ is the zip group of $\mu$ (see Definition~\ref{Def:EmuNorm}) consisting of pairs $(p_-, p_+)$ sharing the same Levi components. Further, the right action of $\mathsf{E}_{\mu}$ on $G^2$ is defined as
		\[
		(g, h) \cdot (p_-, p_+) = (p_-^{-1}g, p_+^{-1}h).
		\]
		For each integer $n \geq 1$, we have the canonical isomorphism
		\[ 
		\prescript{}{1}{\mathcal{C}_1^{\mu}} \ \cancong \ \prescript{}{1}{\mathcal{C}_1^{n\mu}}, \quad g \mu(t) h \mapsto g \mu(t)^n h.
		\]
		In case $\mu$ is minuscule, the canonical projection from $\mathcal{C}_1^\mu := \mathcal{K}\mu(t) \mathcal{K}/\mathcal{K}_1$ to the open affine Schubert variety $\mathrm{Gr}_\mu := \mathcal{K}\mu(t) \mathcal{K}/\mathcal{K}$ factors through the quotient $\prescript{}{1}{\mathcal{C}_1^\mu}$.
	\end{customthm}
	
	This theorem summarizes the results presented in Theorem~\ref{Thm:keyRep}, Proposition~\ref{Prop:qaffine}, and Lemma \ref{Lem:ConnSchbt}. The technical foundation for establishing the explicit representability $\prescript{}{1}{\mathcal{C}_1^\mu} \cong \mathsf{E}_{\mu} \backslash G^2$ consists of Lemmata~\ref{Lem:IntegLem} and \ref{Lem:keyInclu}, wherein the interrelation of the unipotent groups $U_\pm$ and the pro-unipotent group $\mathcal{K}_1$ is discussed. An equivalent way of stating the strong quasi-affineness of $\mathsf{E}_{\mu} \backslash G^2$ is that the closed subgroup $\mathsf{E}_{\mu} \subseteq G^2$ is a Grosshans subgroup; see \S \ref{S:Grosshans} for details. Proposition~\ref{Prop:qaffine} records a proof of the (known) fact that $\mathsf{E}_{\mu} \backslash G^2$, and hence also $\prescript{}{1}{\mathcal{C}_1^\mu}$, is $G^2$-spherical. A more fancy way of stating our representability result is that the coarse moduli space of the quotient stack 
    \[
    [\mathcal{K}_1 \backslash \mathcal{K} \mu(t) \mathcal{K} / \mathcal{K}_1],
    \]
    which is the moduli stack of 1-1 truncated modifications of type $\mu$, exists as a scheme; see Remark \ref{Rmk:trunSch}. The zip group $\mathsf{E}_{\mu}$ 
    here is a special case of the zip groups introduced in \cite{PinkWedhornZiegler1}.
	
	It turns out that the variety of our concern, $\prescript{}{1}{\mathcal{C}_1^\mu}$, has various incarnations in different contexts: the varieties $\mathsf{E}_{\mu} \backslash G^2$ appear in the extant literature under different names (up to isomorphism), e.g., Lusztig's varieties \cite{LuszParobCharShvII}, coset varieties \cite{PinkWedhornZiegler1}, boundary degenerations \cite{SakVenPeriods}, etc. The connection with Lusztig's varieties and coset varieties will be detailed in \S\,\ref{S:Lusztig}.
	
	Assume now that $k$ is a perfect field of characteristic $p $. Let $\mathcal{G}$ be a reductive group over $W(k)$ and $\mu: \mathbb{G}_{m, W(k)} \to \mathcal{G}$ a cocharacter of $\mathcal{G}$ defined over $W(k)$. Write $G = \mathcal{G}_k$ and abuse $\mu$ to denote its base change to $k$. Denote the mixed characteristic counterpart of $\mathcal{K}$ and $\mathcal{K}_1$ by $K = L^+ \mathcal{G}$ and $K_1$, respectively. One may also consider the mixed characteristic version of $\prescript{}{1}{\mathcal{C}_1^\mu}$, a sheaf over the category $\mathrm{Perf}_k$ of perfect $k$-algebras, defined as (see \S \ref{S:MixRep} for details)
	\[\prescript{}{1}{\mathrm{C}_1^\mu} := K_1 \backslash K \mu(p) K / K_1.\]
	
	\begin{customthm}{$\text{A}'$}\label{MainThm:perfect}
		The sheaf $\prescript{}{1}{\mathrm{C}_1^{\mu}}$ over $\mathrm{Perf}_k$ is perfectly represented by $\prescript{}{1}{\mathcal{C}_1^\mu}$. In other words, we have the canonical isomorphism 
		\[\prescript{}{1}{\mathrm{C}_1^\mu} \cancong (\prescript{}{1}{\mathcal{C}_1^\mu})^{\mathrm{perf}} \cong (\mathsf{E}_{\mu} \backslash G^2)^{\mathrm{perf}}.\]
	\end{customthm}
	
	This is Theorem~\ref{Thm:eqvsmix}. With slight modifications, the proof of Theorem~\ref{MainThm:perfect} is identical to that of Theorem~\ref{MainThm:rep}. In fact, Viehmann also proved the homeomorphism
	\begin{equation}\label{Eq:homeo4}
		{}^J W \overset{\epsilon_3}{\cong} \prescript{}{1}{\mathrm{C}_1^{\sigma(\mu)}}(\bar{k})/\mathrm{Ad}_\sigma K(\bar{k}) = \prescript{}{1}{\mathrm{C}_1^{\sigma(\mu)}}(\bar{k})/\mathrm{Ad}_\sigma G(\bar{k}), 
	\end{equation}
	by running parallel arguments~\cite{ViehmannTruncation1}, where $\bar{k}$ denotes an algebraic closure of $k$. However, the presence of ${}^J W$ in the two homeomorphisms above is merely a coincidental circumstance, as the right-hand sides of \eqref{Eq:homeo1} and \eqref{Eq:homeo4} were not directly related a priori. With Theorem~\ref{MainThm:perfect}, we can now provide a geometric explanation:
	\begin{center}
		\emph{The right-hand sides of \eqref{Eq:homeo1} and \eqref{Eq:homeo4} are already identified before taking $G(k)$-$\sigma$ conjugacy classes.}
	\end{center}
	We refer to Corollary~\ref{Cor:Comp} for more detailed comparisons. An immediate consequence of Theorem~\ref{MainThm:perfect} is the algebraicity of $[\prescript{}{1}{\mathcal{C}_1^\mu}/\mathrm{Ad}_{\sigma} G]$.
	
	\begin{corollary}
		The quotient stack $[\prescript{}{1}{\mathcal{C}_1^\mu}/\mathrm{Ad}_{\sigma} G]$ is a smooth Artin $k$-stack of dimension $0$.  
	\end{corollary}
	
	This provides a reasonable geometrization of the right-hand side of \eqref{Eq:homeo1}. We call this stack the \emph{(coarse) moduli stack of 1-1 truncated local $G$-shtukas} of type $\mu$. This may be seen as a function field analogue of ${G\text{-}\mathsf{Zip}}^{\mu}$. In fact, it turns out to be more than a mere philosophical cousin; see Theorems~\ref{MainThm:B} and \ref{MainThm:com} below.
	
	\begin{corollary} 
		The embedding $\mathrm{Emb}_{\mu}: G \to \mathcal{C}^{\mu}, g \mapsto g \mu(t)$ induces an immersion of $k$-schemes
		\begin{align*}
			\tilde{\alpha}: G/U_- \ \cong\ G \mu(t) \mathcal{K}_1 / \mathcal{K}_1 \ \subseteq\ \mathcal{C}^{\mu}_1 := \mathcal{C}^\mu / \mathcal{K}_1.
		\end{align*} 
		Moreover, the composition $ \alpha: G/U_{-} \to \mathcal{C}_1^{\mu} \xrightarrow{\mathrm{pr}} \prescript{}{1}{\mathcal{C}_1^\mu}$ is a closed immersion. The canonical projection from $G \mu(t) \mathcal{K}_1 / \mathcal{K}_1 $ to $\prescript{}{1}{\mathcal{C}_1^\mu}$ is also a closed immersion.
	\end{corollary}	

	We direct the reader to Corollary~\ref{Cor:Embed} for details and for a dual version that embeds $  U_+ \backslash G$ into $\prescript{}{1}{\mathcal{C}^\mu}$ and $\prescript{}{1}{\mathcal{C}_1^\mu}$.
    
    Thanks to the explicit isomorphism $\prescript{}{1}{\mathcal{C}_1^\mu} \cong \mathsf{E}_{\mu} \backslash G^2$, abstract nonsense arguments are essentially sufficient to show that $[\prescript{}{1}{\mathcal{C}_1^\mu}/\mathrm{Ad}_\sigma G]$ is also a zip stack. The theorems presented below can also be seen as corollaries to Theorem~\ref{MainThm:rep}.
	
	\begin{customthm}{B}\label{MainThm:B}
		We have an (explicitly defined) isomorphism of algebraic $k$-stacks
		\begin{equation}\label{Eq:zipisom2}
			[G/\mathsf{R}_{\sigma} \mathsf{E}_{\mu}] \cong [\prescript{}{1}{\mathcal{C}_1^\mu}/\mathrm{Ad}_{\sigma} G ], \quad g \mapsto \mu(t)g,  
		\end{equation}
		where $\mathsf{R}_{\sigma} \mathsf{E}_{\mu}$ denotes the $\mathsf{E}_{\mu}$ action on $G$ by \emph{right partial Frobenius multiplication}
		\[
		g \cdot (p_-, p_+) = p_+^{-1} g \sigma(p_-).
		\]
	\end{customthm} 
	
	Refer to \S\,\ref{S:zipstack} and especially Lemma~\ref{Lem:ManyIsom2} for more details. We show there that more general quotient stacks of $\prescript{}{1}{\mathcal{C}_1^\mu}$ admit zip stack descriptions. An interesting case occurs when $\sigma$ is formally replaced by the identity map $\mathrm{id}: G \to G$. In other words, we are dealing with the quotient stack $[\prescript{}{1}{\mathcal{C}_1^\mu}/\mathrm{Ad} G]$, which does not involve the Frobenius operator $\sigma$ and thus makes sense for any base field. Similar abstract arguments as mentioned above (in fact, slightly simpler ones) will help us to show that 
	\begin{equation}\label{Eq:zipisom1}
		[\prescript{}{1}{\mathcal{C}_1^\mu}/\mathrm{Ad} G] \cong [G/\mathsf{E}_{\mu}].
	\end{equation}
	
	In \S \ref{S:prozip}, we define a pro-zip group $\mathcal{E}_\mu \subseteq \mathcal{L}^+G \times \mathcal{L}^+G$ and formulate the following isomorphisms: 
	\begin{equation}\label{Eq:IntIsom}
		\mathcal{E}_\mu \backslash (\mathcal{L}^+G \times \mathcal{L}^+G) \cong \mathcal{C}^{\mu},  \quad [\mathcal{L}^+G/\mathsf{R}_{\sigma} \mathcal{E}_\mu] \cong [\mathcal{C}^{\mu} /\mathrm{Ad}_{\sigma} \mathcal{L}^+G], \quad [\mathcal{L}^+G/\mathcal{E}_\mu] \cong [\mathcal{C}^{\mu} /\mathrm{Ad}\mathcal{L}^+G].
	\end{equation}
      The last of these leads to an interpretation of local $G$-shtukas as $\mathcal{L}^+ G$-zips. The left-hand sides of the three isomorphisms in~\eqref{Eq:IntIsom} has the advantage of being integral, thereby making their reduction modulo \(t^n\) well-defined. For instance, the reduction modulo $t$ of the isomorphisms in~\eqref{Eq:IntIsom} recovers those in~\eqref{Eq:KeyIsom}, \eqref{Eq:zipisom2}, and \eqref{Eq:zipisom1}.

	\begin{customthm}{C}\label{MainThm:com}
		We have a chain of morphisms between algebraic $k$-stacks 
		\[  \cdots \longrightarrow [\prescript{}{1}{\mathcal{C}_1^\mu}/\mathrm{Ad}_{\sigma} G] \longrightarrow {G\text{-}\mathsf{Zip}}^{\mu} \longrightarrow [\prescript{}{1}{\mathcal{C}_1^{\sigma(\mu)}}/\mathrm{Ad}_{\sigma} G] \longrightarrow \cdots \]
		These morphisms are induced by ``partial Frobenii" in a suitable sense. They induce homeomorphisms on the underlying topological spaces. Moreover, the following diagram commutes 
		\[
		\begin{tikzcd}[row sep=2.5em,column sep=1.5em]
			{} & \arrow{dl}{\cong}[swap]{\epsilon_2} {}^J W \arrow{dr}{\epsilon_1}[swap]{\cong} \\
			\protect{|{G\text{-}\mathsf{Zip}}^{\mu}|} \arrow{rr}{\cong} && \protect{[\prescript{}{1}{\mathcal{C}_1^{\sigma(\mu)}}/\mathrm{Ad}_{\sigma} G](k)}
		\end{tikzcd}.
		\]
	\end{customthm}
	The above theorem is the content of Proposition \ref{Prop:ManyIsom3}, Corollary \ref{Cor:Comp0} and \ref{Cor:Comp}, where details and (slightly) more general results are discussed. The proof of Theorem~\ref{MainThm:com} is largely a matter of abstract reasoning once again; however, establishing the homeomorphisms $\epsilon_1$ and $\epsilon_2$ is non-trivial. Our results indicate that they are indeed equivalent. Theorem~\ref{MainThm:com} also shows that ${G\text{-}\mathsf{Zip}}^{\mu}$ and its function field counterpart $[\prescript{}{1}{\mathcal{C}_1^\mu}/\mathrm{Ad}_{\sigma} G]$ are intertwined by partial Frobenii.

    These findings deepen our understanding of zip period maps for Shimura varieties, leading to improvements in our previous work \cite{Yan18} (revisited in \S \ref{S:App}). Moreover, the main results of this paper lay the foundation for \cite{yan25intperdmap}.

	\subsection{An open problem}
	Suppose that the reductive group $G$ is split over $k$, with no constraints on the characteristic of $k$, and we have chosen a Borel pair $(B, T)$ with the maximal torus $T$ split. Then we have the Cartan decomposition of $\mathcal{L}G = \bigsqcup_{\mu \in X_*(T)_{\mathrm{dom}}} \mathcal{C}^\mu$ into locally closed subschemes. The closure of $\mathcal{C}^\mu$, denoted by $\mathcal{C}^{\leq \mu}$, is set-theoretically equal to the union of $\mathcal{C}^{\mu'}$ with $\mu' \leq \mu$. Motivated by the fact that the (closed) affine Schubert variety $\mathrm{Gr}_{\leq \mu}:=\mathcal{C}^{\leq \mu}/\mathcal{K}$ is represented by a normal $k$-scheme, we make the following conjecture regarding the representability of the fpqc $k$-sheaf defined by
	\[ \prescript{}{1}{\mathcal{C}_1^{\leq \mu}} := \mathcal{K}_1 \backslash \mathcal{C}^{\leq \mu} / \mathcal{K}_1. \]
	\begin{conjecture}\label{conjectureIntro}
		The fpqc sheaf $\prescript{}{1}{\mathcal{C}_1^{\leq \mu}}$ is represented by a normal $k$-scheme. 
	\end{conjecture}
	The conjecture holds true for minuscule $\mu$, since in that case we have $\prescript{}{1}{\mathcal{C}_1^\mu} = \prescript{}{1}{\mathcal{C}_1^{\leq \mu}}$. Demonstrating the validity of this conjecture would imply that the larger space $\mathcal{K}_1 \backslash \mathcal{L}G / \mathcal{K}_1$ is an ind-scheme. 
	
   \subsection{Organization of the Paper}	
   Section \ref{S:Preliminary} presents preliminary material, including foundational facts on loop groups and group action data. Section \ref{S:UnipGps} examines the links between unipotent and pro-unipotent groups, establishing key technical lemmas. In Section \ref{S:RepCIImu}, we prove the central representability results in Theorem \ref{MainThm:rep} and Theorem \ref{MainThm:perfect}, and conclude with Conjecture \ref{conjectureIntro}. Section \ref{S:Properties} discusses various properties of the variety, with some results summarized in Theorem \ref{MainThm:rep}. Section \ref{S:ShtukAsZip} compares the moduli stack of 1-1 truncated local $G$-shtukas with $G$-zip stacks, leading to Theorem \ref{MainThm:com}, and in the last subsection we interpret the moduli stack of local $G$-shtukas of fixed type $\mu$ as a pro-zip stack. Finally, Section \ref{S:App} revisits our earlier work on zip period maps \cite{Yan24} \cite{Yan18} and highlights improvements based on the main results of this paper.

\subsection{Notation and conventions}\label{NotatConven}
\begin{enumerate}
\item We write \cancong\ for canonical isomorphisms and $\mathrm{Pr}$ for canonical projections.

\item For a field $k$ of characteristic $p$, we denote by $\sigma: k \to k$ the absolute (i.e.\ $p$-power) Frobenius endomorphism of $k$. If $X$ is a scheme (or any geometric object) over $k$, its base change along $\sigma$ is written $X^\sigma$, and we denote the relative Frobenius of $X$ by $\sigma: X \to X^\sigma$.

\item Given a base scheme $S$, we define stacks over $S$ as in \cite{ChamAlge}, namely over the \'etale site $(\mathrm{Aff}/S)_{\mathrm{et}}$ of affine $S$-schemes.

\item We consider all group actions to be right actions. For a group $G$ and a subgroup $H \subseteq G$, the subgroup $H$ acts on $G$ on the right by left multiplication; specifically, $h \in H$ sends $g \in G$ to $h^{-1} g$. The quotient space is denoted $H \backslash G$.

\item For an algebraic group $G$, we denote the equicharacteristic loop group and associated positive group by $\mathcal{L}G, \mathcal{K} = \mathcal{L}^+G$, and their mixed characteristic counterparts by $LG, K = L^+G$. For example, $\prescript{}{1}{\mathrm{C}_1^\mu}$ indicates the mixed characteristic analogue of $\prescript{}{1}{\mathcal{C}_1^\mu}$.
\end{enumerate}

\subsection{Acknowledgements} 
  This paper was partly inspired by Fabrizio Andretta's question regarding my PhD thesis, and I am grateful for his input. I thank Yigeng Zhao and Westlake University for providing an exceptional environment during my 2019 visit, where much of this work was developed. Special thanks to Sian Nie for insightful group theory discussions, and to Jingren Chi, Ulrich Goertz, Xuhua He, Xu Shen, Liang Xiao, Chao Zhang, and Xinwen Zhu for their helpful remarks. I also thank the anonymous referees for their valuable suggestions, which greatly improved this manuscript.

    \section{Preliminaries on Loop Groups and group action data} \label{S:Preliminary}
   This section collects basic definitions and results on loop groups and on group action data—two largely distinct topics—while also establishing the notation and conventions that will be used throughout the subsequent discussion.

\subsection{Loop Groups}\label{S:loopgp}
Let \( k \) be a field and \( G \) a reductive group over \( k \). In this subsection, we recall the definitions of the loop group \( \mathcal{L} G \) and the positive loop group \( \mathcal{L}^+ G \), along with some of their basic properties. The representability of a subfunctor of \( \mathcal{L} G \) is proven shortly thereafter.

The \emph{loop group} \( \mathcal{L} G \) is the fpqc sheaf of groups whose \( A \)-valued points for a \( k \)-algebra \( A \) are given by \( \mathcal{L} G(A) = G(A((t))) \), where \( A((t)) \) is the ring of formal Laurent series with coefficients in \( A \).

The \emph{affine Grassmannian} \( \mathrm{Gr} = \mathrm{Gr}_G \) of \( G \) is the fpqc sheafification of the presheaf assigning the quotient \( \mathcal{L} G(R)/\mathcal{L}^+ G(R) \) to every \( k \)-algebra \( R \). For each integer $n\geq 1$, we denote \( \mathcal{L}_{n} G \) as the Weil restriction \( \mathrm{Res}_{(k[t]/t^n)/k} G \), which is again a connected reductive group over \( k \). Let \( \mathcal{K}_n \) be the kernel of the reduction modulo \( t^n \) map from \( \mathcal{L}^+ G \) to \( \mathcal{L}_n G \). When necessary, to avoid confusion regarding the group \( G \), we will use \( \mathcal{L}^n G := \mathcal{K}_n \). The canonical embedding \( G \hookrightarrow \mathcal{L}^+ G \) induces a decomposition \( \mathcal{L}^+ G = G \ltimes \mathcal{K}_1 \).

For a \( k \)-scheme \( T \) and an element \( g \in \mathcal{L}^+ G(T) \), we denote by \( \bar{g} \) its image in \( G(T) \) along the canonical projection \( \mathrm{pr} : \mathcal{L}^+ G \xrightarrow{t\mapsto 0} G \).
The next theorem summarizes some standard facts concerning \( \mathcal{L}^+ G \) and \( \mathcal{L} G \). 

\begin{theorem}\label{Thm:factsloopGp} The following holds:              
	\begin{enumerate}
		\item\label{I:plim} The positive loop group \( \mathcal{L}^+ G \) is a projective limit of smooth affine group \( k \)-schemes (the limit taken in the category of fpqc sheaves over \( k \)):
		\[
		\mathcal{L}^+ G = \lim_{n \geq 1} \mathcal{L}_{n} G.
		\]
		\item\label{I:red} The positive loop group is represented by a formally smooth affine group scheme over \( k \) of infinite dimension (unless \( G \) is trivial). In particular, \( \mathcal{L}^+ G \) is reduced.  
		\item\label{I:ind} The loop group \( \mathcal{L} G \) is an ind-affine, ind-finite type ind-scheme, i.e., \( \mathcal{L} G = \lim_{i \geq 1} G^i \) (as fpqc sheaves), where the fpqc sheaf \( G^i \) is represented by an affine scheme of finite type over \( k \), and the transition map \( G^i \hookrightarrow G^{i+1} \) is a closed immersion for each \( i \geq 1 \). 
		\item\label{I:indproj} The affine Grassmannian \( \mathrm{Gr} \) is represented by an ind-projective ind-scheme, i.e., in the category of fpqc sheaves over \( k \):
		\[
		\mathrm{Gr} = \lim_{i \geq 1} \mathrm{Gr}^i,
		\]
		where the fpqc sheaf \( \mathrm{Gr}^i \) is represented by a projective scheme over \( k \), and the transition map \( \mathrm{Gr}^i \hookrightarrow \mathrm{Gr}^{i+1} \) is a closed immersion for each \( i \). 
	\end{enumerate}
\end{theorem}

\begin{proof}
	For statement \eqref{I:plim}, the nontrivial part is to show that the Weil restriction \( \mathrm{Res}_{(k[t]/t^n)/k} G \) is represented by a smooth, connected affine group scheme of finite type over \( k \) -- a standard fact. Interested readers may refer to \cite{PseudoReductiveGroups}[A.5].

	In regards to statements \eqref{I:red}, \eqref{I:ind}, and \eqref{I:indproj}, we refer to \cite{TwistedLoopGroupsRapoport&Pappas} for a proof. However, let us hint for later purposes how to obtain \( G^i \) and \( \mathrm{Gr}^i \) in the expressions \( \mathcal{L} G = \lim_{i \geq 1} G^i \) and \( \mathrm{Gr} = \lim_{i \geq 1} \mathrm{Gr}^i \). The basic idea is to embed \( G \) into \( \mathrm{GL}_{n, k} \), reducing the problem to showing that \( \mathcal{L}^+ \mathrm{GL}_{n, k} \) is represented by an affine group scheme and that \( \mathcal{L}^+ G \hookrightarrow \mathcal{L}^+ \mathrm{GL}_{n, k} \) is a closed embedding. The case \( G = \mathrm{GL}_{n, k} \) has been essentially dealt with in the original work \cite{Beauville1994}, where \( k \) was taken to be \( \mathbb{C} \). For \( i \geq 0 \), let \( \mathrm{GL}_{n, k}^{i} \) denote the subfunctor of \( \mathcal{L} \mathrm{GL}_{n, k} \) defined by
	\[
	\mathrm{GL}_{n, k}^{i}(R) = \left\{ M \in \mathrm{GL}_n(R((t))) : M, M^{-1} \in t^{-i} \mathrm{M}_{n}(R[[t]]) \right\}.
	\]
	One can show that each \( \mathrm{GL}_{n, k}^i \) is represented by an affine scheme (which is not a group scheme unless \( i = 0 \)), and the transition morphisms are closed immersions. Hence, we have \( \mathcal{L} \mathrm{GL}_{n, k} = \bigcup_{i=0}^{\infty} \mathrm{GL}_{n, k}^{i} \). For the general case, we set
	\begin{equation}\label{Embedding of loop group}
		G^i = \mathcal{L} G \cap \mathrm{GL}_{n, k}^{i}, \quad \mathrm{Gr}^i = G^i / \mathcal{L}^+ G.   
	\end{equation}
\end{proof}

\begin{remark}
	From now on, we define \( G^i \) and \( \mathrm{Gr}^i \) as in \eqref{Embedding of loop group}, by choosing an embedding \( G \hookrightarrow \mathrm{GL}_{n, k} \) whenever the expression \( \mathcal{L} G = \lim_{i \geq 1} G^i \) is needed in later applications. The obtained \( G^i \) exhibit useful properties:
	\begin{enumerate}
		\item \( G^0 = \mathcal{L}^+ G \),
		\item Each \( G^i \) is stable under the action of \( \mathcal{L}^+ G \) on \( \mathcal{L} G \) via left or right multiplication, i.e.,
		\[
		\mathcal{L}^+ G \cdot G^i \cdot \mathcal{L}^+ G = G^i.
		\]
		\item For each \( k \)-algebra \( R \), \( \mathcal{L} G(R) = \bigcup_{i=0}^{\infty} G^i(R) \).
	\end{enumerate}
\end{remark}

For an element \( x \in \mathcal{L} G(k) \), let \( \mathcal{C}^x \subseteq \mathcal{L} G \) be the fpqc sheafification of the presheaf assigning the subset \( \mathcal{K}(R)x\mathcal{K}(R) \) of \( \mathcal{L} G(R) \) to any \( k \)-algebra \( R \). The next lemma addresses the representability of \( \mathcal{C}^x \), concerning the representability of an orbit of a group action (\( \mathcal{L}^+ G \times \mathcal{L}^+ G \) acts on \( \mathcal{L} G \) by simultaneous left and right multiplication), where both the group scheme and the scheme being acted upon are infinite-dimensional. The validity of this lemma is known to experts, but we sketch a proof here to compensate for the lack of a precise reference.

\begin{lemma}
	The sheaf \( \mathcal{C}^x \) is represented by a formally smooth \( k \)-scheme. 
\end{lemma}
\begin{proof}
	Take \( G^i \) as in \eqref{Embedding of loop group} so that \( \mathcal{L} G = \lim_{i \geq 1} G^i \). Then, we have \( x \in G^i(k) \) for some \( i \). Let \( \bar{x} \in \mathrm{Gr}^i(k) \) be the image of \( x \) under the canonical projection \( \mathrm{pr}: G^i \twoheadrightarrow \mathrm{Gr}^i \), and consider the \( \mathcal{L}^+ G \) action on \( \mathcal{C}^i \) (resp. \( \mathrm{Gr}^i \)) via left multiplication. From the definition of \( G^i \), it can be verified that, reducing to the case \( G = \mathrm{GL}_{n,k} \), the \( \mathcal{L}^+ G \)-action on \( \mathrm{Gr}^i \) factors through the canonical projection \( \mathcal{L}^+ G \twoheadrightarrow \mathcal{L}_{2i} G \).

	Let \( \mathrm{orb}_{\bar{x}}: \mathcal{L}^+ G \twoheadrightarrow \mathcal{L}_{2i} G \to \mathrm{Gr}^i \) be the orbit map of \( \bar{x} \). Notice that \( \mathrm{Gr}^i \) is a projective scheme over \( k \) and \( \mathcal{L}_{2i} G \) is a smooth connected linear algebraic group over \( k \). From the theory of smooth linear algebraic groups acting on algebraic schemes (see \cite[Proposition~7.17]{MilneAlgebraicGroups}), we know that \( \mathrm{orb}_{\bar{x}}(\mathcal{L}^+ G) \) is a locally closed subset of \( \mathrm{Gr}^i \), and the \( \mathcal{L}_{2i} G \)-orbit \( O_{\bar{x}} \subseteq \mathrm{Gr}^i \) (the reduced subscheme of \( \mathrm{Gr}^i \) with \( \mathrm{orb}_{\bar{x}}(\mathcal{L}^+ G) \) as its underlying topological space) is a smooth connected separated algebraic scheme over \( k \). Let \( O_x \subseteq G^i \) be the pullback of \( O_{\bar{x}} \) along \( \mathrm{pr} \). The orbit map \( \mathrm{orb}_{x}: \mathcal{L}^+ G \to G^i \) of \( x \) factors through \( O_x \), thereby ensuring \( O_x \supseteq \mathcal{C}^x \) as fpqc sheaves. Finally, it is evident that fpqc locally, each local section of \( O_x \) arises from \( \mathcal{C}^x \). 
\end{proof}

\subsection{Frobenii on the Loop Group}\label{S:TwoFrob}
Let \( k \), \( G \), and \( \mu \) be as in the previous subsection, with the additional assumption that \( k \) is of characteristic \( p \). In this section, we examine the Frobenii of the loop group \( \mathcal{L}G \). Although our discussion extends beyond the immediate needs of this paper, it is intended for future reference.

Let \( \sigma: k \to k \) denote the absolute Frobenius of \( k \), and by abuse of notation, we write \( \sigma: \mathrm{Spec} k \to \mathrm{Spec} k \) for the corresponding scheme map. Let \( G^{\sigma} \) be the pullback of \( G \) along \( \sigma \). Suppose \( G = \mathrm{Spec} A \) for some \( k \)-algebra \( A \). For any \( k \)-algebra \( R \), we have \( \mathcal{L} G(R) = \mathrm{Hom}_{k}(A, R((t))) \), where the \( k \)-algebra structure on \( R((t)) \) is given by \( k \to R \hookrightarrow R((t)) \). For \( G^{\sigma} \), we have \( \mathcal{L} G^{\sigma}(R) = \mathrm{Hom}_{k}(A, {}^{\sigma}R((t))) \), where \( {}^{\sigma}R((t)) \) denotes \( R((t)) \) with a \( k \)-algebra structure defined via \( k \xrightarrow{\sigma} k \to R \hookrightarrow R((t)) \).

For each \( R \), let \( \sigma: R((t)) \to R((t)) \) be the ring endomorphism that applies the absolute Frobenius to \( R \) while fixing \( t \). The map \( \mathcal{L} G(A) \to \mathcal{L} G^{\sigma}(A) \), sending \( f \in \mathrm{Hom}_{k}(A, R((t))) \) to \( \sigma \circ f \in \mathrm{Hom}_{k}(A, {}^{\sigma}R((t))) \), is functorial in \( R \) and hence defines a \( k \)-homomorphism of loop groups:
\begin{equation}\label{Sigofloopgp}
	\sigma = \sigma_{\mathcal{L} G}: \mathcal{L} G \to \mathcal{L} G^{\sigma}.
\end{equation}

Alternatively, we can express \( \sigma: \mathcal{L} G \to \mathcal{L} G^{\sigma} \) as follows. Write \( \mathcal{L} G = \lim_{i \geq 1} G^i \) by choosing an embedding \( G \hookrightarrow \mathrm{GL}_{n, k} \) as in (\ref{Embedding of loop group}). For each \( k \)-algebra \( R \), the equality
\[
\mathcal{L} G^{\sigma}(R) = \mathcal{L} G({}^{\sigma} R) = \bigcup_{i \geq 0} G^i({}^{\sigma} R) = \bigcup_{i \geq 0} G^{i, \sigma}(R)
\]
implies that \( \mathcal{L} G^{\sigma} = \lim_{i \geq 0} G^{i, \sigma} \). The compatible system of relative Frobenii \( \sigma: G^i \to G^{i, \sigma} \) defines a \( k \)-homomorphism \( \mathcal{L} G \to \mathcal{L} G^{\sigma} \) that coincides with \( \sigma_{\mathcal{L} G} \) in (\ref{Sigofloopgp}).

Furthermore, the formation of \( \mathcal{L} G \) is clearly functorial in \( G \), thus the relative Frobenius \( \sigma: G \to G^{\sigma} \) gives rise to a \( k \)-homomorphism of loop groups:
\[
\varphi = \varphi_{\mathcal{L} G} = \mathcal{L}(\sigma): \mathcal{L} G \to \mathcal{L} G^{\sigma},
\]
 which sends \( f \in \mathrm{Hom}_{k}(A, R((t))) \) to \( \varphi \circ f \in \mathrm{Hom}_{k}(A, {}^{\sigma}R((t))) \). Here, \( \varphi: R((t)) \to R((t)) \) denotes the absolute Frobenius on \( R((t)) \) (thus sending \( t \) to \( t^p \)).

Both homomorphisms of group functors \( \sigma, \varphi: \mathcal{L} G \to \mathcal{L} G \) stabilize \( \mathcal{L}^+ G \) and thus induce homomorphisms of group \( k \)-schemes:
\[
\sigma, \varphi: \mathcal{L}^+ G \to \mathcal{L}^+ G^{\sigma}.
\]

These can also be defined from finite levels. For any \( n \geq 1 \) and any \( k \)-algebra \( R \), the ring endomorphisms \( \sigma: R[[t]] \to R[[t]] \) and \( \varphi: R[[t]] \to R[[t]] \) induce ring endomorphisms \( \sigma: R[t]/t^{n} \to R[t]/t^{n} \) and ring maps \( R[t]/t^{n} \to R[t]/t^{np} \). The maps \( \mathcal{L}_{n} G(A) \to \mathcal{L}_{n} G(A) \) (resp. \( \mathcal{L}_{n} G(A) \to \mathcal{L}_{np} G(A) \)) sending \( f \in \mathrm{Hom}_{k}(A, R[t]/t^{n}) \) to \( \sigma \circ f \in \mathrm{Hom}_{k}(A, {}^{\sigma}(R/t^{n})) \) (resp. to \( \varphi \circ f \in \mathrm{Hom}_{k}(A, {}^{\sigma}(R/t^{np})) \)) are functorial in \( R \), and thus induce homomorphisms of \( k \)-group schemes \( \sigma: \mathcal{L}_{n} G \to \mathcal{L}_{n} G \) (resp. \( \varphi: \mathcal{L}_{n} G \to \mathcal{L}_{np} G \)), fitting into the following commutative diagrams, where vertical arrows represent natural reduction maps:
\begin{align}
	\begin{array}{cc}
		\xymatrix{\mathcal{L}^+ G \ar[r]^{\sigma }\ar[d] & \mathcal{L}^+ G^{\sigma}\ar[d]\\
			\mathcal{L}_{n} G\ar[r]^{\sigma}& \mathcal{L}_{n} G^{\sigma},} \quad \text{resp.} \quad
		\xymatrix{\mathcal{L}^+ G \ar[r]^{\varphi }\ar[d] & \mathcal{L}^+ G^{\sigma}\ar[d]\\
			\mathcal{L}_{n} G\ar[r]^{\varphi}& \mathcal{L}_{np} G^{\sigma}.}
	\end{array}
\end{align}
In particular, we have
\[
\sigma(\mathcal{L}^{n} G) \subseteq \mathcal{L}^{n} G^{\sigma}, \quad \varphi(\mathcal{L}^{n} G) \subseteq \mathcal{L}^{np} G^{\sigma}. 
\]

\begin{example}
	Let \( G = \mathbb{G}_{m,k} \) and \( R \) a \( k \)-algebra. Then, for \( f = \sum a_i t^i \in \mathcal{L} \mathbb{G}_m(R) \), we have:
	\[
	\sigma(f) = \sum a_i^p t^i, \quad \varphi(f) = \sum a_i^p t^{pi}.
	\]
\end{example}

From now on, we assume that our reductive group \( G \) is defined over \( \mathbb{F}_p \subseteq k \). Thus, we have the canonical isomorphism \( G \cong G^{\sigma} \). Consequently, from the preceding discussion, we have two Frobenius endomorphisms \( \sigma, \varphi: \mathcal{L} G \longrightarrow \mathcal{L} G \) of \( \mathcal{L} G \), which preserve \( \mathcal{L}^+ G \), and both are homomorphisms of \( k \)-group functors. Let \( \mu: \mathbb{G}_{m,k} \to G \) be a cocharacter of \( G \). Define
\[
\sigma(\mu) = \mu^{\sigma}: \mathbb{G}_{m,k} \to G, 
\]
to be the base change along \( \sigma: k \to k \) of \( \mu \) and define 
\(
\varphi(\mu) = \mu^{\varphi}=\sigma \circ \mu, 
\)
with $\sigma: G\to G$ the relative Frobenius of $G$. One verifies that
\begin{equation} \label{Eq:TwoFrobeniRelation}
  \varphi(\mu) = p \sigma(\mu).    
\end{equation}

The following lemma follows from direct verification (e.g., by choosing an embedding \( G \hookrightarrow \mathrm{GL}_{n, k} \)). 

\begin{lemma}\label{Lem:FrobvsMu}
	Let \( G \) and \( \mu \) be as above. Assume further that \( G \) is defined over \( \mathbb{F}_p \). Then 
	\begin{enumerate}
		\item We have \( \mu^\sigma(t) = \sigma(\mu(t)) \).
		\item The diagram below is commutative: 
		\begin{equation*}
			\xymatrix{& \mathcal{L} G \ar[dr]^{\varphi_{\mathcal{L} G}}&\\
				\mathcal{L} \mathbb{G}_{m,k} \ar[rr]^{\mathcal{L}(\mu^{\varphi})}\ar[ur]^{\mathcal{L}(\mu)}\ar[dr]_{\varphi_{\mathcal{L} \mathbb{G}_{m,k}}} && \mathcal{L} G.\\
				& \mathcal{L} \mathbb{G}_{m,k} \ar[ur]_{\mathcal{L}(\mu^{\sigma})} &}
		\end{equation*}
	\end{enumerate}
    In particular,  \( \mu^{\varphi}(t) = \varphi(\mu(t)) \) holds.
\end{lemma}

\begin{remark}
	The upper triangle in the diagram of Lemma \ref{Lem:FrobvsMu} indicates that \( \mathcal{L}(\mu^{\varphi}) = \varphi_{\mathcal{L} G} \circ \mathcal{L}(\mu) \). In contrast, for \( \mu^\sigma \), the analogous relation \( \mathcal{L}(\mu^\sigma) = \sigma_{\mathcal{L} G} \circ \mathcal{L}(\mu) \) does not hold. However, we still have \( \mu^{\sigma}(t) = \sigma(\mu(t)) \) because \( t \), as a special power series in \( k((t)) \), has coefficients in \( \mathbb{F}_p \) which are fixed by \( \sigma: k \to k \).
\end{remark}
	
\subsection{Group action data}\label{S:GpActData}
	Let $S$ be a base scheme. A \emph{group action datum} over $S$ is a triple $(X,G,\xi)$, where $X$ is an $S$-scheme, $G$ is a group $S$-scheme, and $\xi\colon X\times_SG\to X$ is a right action. Often, we simply write $(X,G)$ when the action is understood. This datum defines the quotient stack $[X/G]$, which assigns to any affine $S$-scheme $T$ the groupoid of pairs $(P,\alpha)$—with $P$ a $G$-torsor over $T$ and $\alpha\colon P\to X$ a $G$-equivariant map—or, equivalently, is the stackification of the prestack $[X/G]^{\mathrm{pre}}$ that sends $T$ to the groupoid with objects $X(T)$ and morphisms $\{g\in G(T): x\cdot g=y\}$.
	
	For another datum $(Y,H,\eta)$, a \emph{morphism of group action data} $(f,\phi):(X,G)\to(Y,H)$ consists of an $S$-morphism $f\colon X\to Y$ and an $S$-group homomorphism $\phi\colon G\to H$ compatible with the actions. This induces a morphism of prestacks $[f,\phi]^{\mathrm{pre}}\colon [X/G]^{\mathrm{pre}}\to[Y/H]^{\mathrm{pre}}$, whose stackification $[f,\phi]\colon [X/G]\to[Y/H]$ (denoted $[f]$) sends $x\in X(T)$ to $f(x)\in Y(T)$ and $g\in\operatorname{Isom}(x,y)$ to $\phi(g)\in\operatorname{Isom}(f(x),f(y))$.
	
    Not every morphism of $S$-stacks $[X/G]\to[Y/H]$ comes from a morphism above. For example, if $f\colon X\to Y$ is an $S$-morphism with $G$ and $H$ acting on $X$ and $Y$, respectively, and there exists a map $\alpha\colon X\times_SG\to H$ satisfying $f(x\cdot g)=f(x)\cdot\alpha(x,g)$ and $\alpha(x,gg')=\alpha(x,g)\,\alpha(x\cdot g,g')$ for all $x\in X$ and $g,g'\in G$, then $f$ induces a morphism $[X/G]\to[Y/H]$ of $S$-stacks.

    Let $(f_1, \phi_1), (f_2, \phi_2): (X, G)\to (Y, H)$ be two morphisms of group action data. The induced morphisms of $S$-stacks $[f_1], [f_2]: [X/G]\to [Y/H]$ are 2-isomorphic, if there exists an $S$-morphism $\beta: X\to H$, such that 
    \(\phi_1(g)\beta (x\cdot g)= \beta(g)\phi_2(g),\)
    for all $g\in G$ and all $x\in X$. 

    The discussion above still makes sense if one replace $X$ by an sheaf on $(\mathrm{Aff}/S)_{\mathrm{et}}$).

	\section{Unipotent Groups versus Pro-Unipotent Groups} \label{S:UnipGps}
Let \( k, G, \mu \) be as in the previous section, without restrictions on the characteristic of \( k \). This section establishes key lemmas forming the technical foundation of this work. These lemmas describe relations between unipotent subgroups of \( G \), defined by $\mu$, and the unipotent subgroup \( \mathcal{K}_1 \) of its positive loop group \( \mathcal{L}^+ G \).

 The pair \( (G, \mu) \) determines the associated algebraic subgroups: the parabolic subgroup \( P_+ = P_{\mu} \) and its unipotent radical \( U_+ \). As group functors they are given by associating a \( k \)-algebra \( A \) with the groups (see \cite[Theorem 13.33]{MilneAlgebraicGroups}),
\begin{align*}
	P_+(A) &= \{ g \in G(A) \mid \lim_{t\to 0} \mu(t)g\mu(t)^{-1} \text{ exists} \},\\
	U_+(A) &= \{ g \in P(A) \mid \lim_{t\to 0} \mu(t)g\mu(t) = 1 \}.
\end{align*}
Similarly, we define the opposite parabolic group \( P_- = P_{\mu^{-1}} \) and its unipotent radical \( U_- = U_{\mu^{-1}} \). Let \( M = \mathrm{Cent}_{G}(\mu) \) be the centralizer in \( G \) of \( \mu \). It is a reductive subgroup of \( G \), satisfying \( M = P_+ \cap P_- \). From the definitions of \( U_{\pm} \subseteq P_{\pm} \), we deduce that if \( G \) is embedded into another reductive group \( G' \) over \( k \), then we have  
\begin{equation}\label{Eq:cutParab}
	P_+ = P'_+ \cap G, \quad P_- = P'_- \cap G, \quad  M = M' \cap G,
\end{equation}
where \( U'_{\pm} \subseteq P'_{\pm} \) are the counterparts of \( U_{\pm} \subseteq P_{\pm} \) for the cocharacter induced from \( \mu \) in \( G' \). 

Given a \( k \)-algebra \( A \) and \( p_+ \in P_+(A) \), using the decomposition \( P_+ = U_+ \rtimes M \), we can express \( p_+ = u_+ m \) with \( u_+ \in U_+(A) \) and \( m \in M(A) \). Here, \( u_+ \) is called the \emph{unipotent radical component} of \( p_+ \) and \( m \) the \emph{Levi component} of \( p_+ \). As in \S \ref{S:Preliminary}, for an element \( g \in \mathcal{L}^+ G(A) \), we denote by \( \bar{g} \in G(A) \) the image of \( g \) under the canonical projection \( \mathrm{pr}: \mathcal{L}^+ G(A) = G(A[[t]]) \xrightarrow{t\mapsto 0} G(A) \).

The following lemma defines the basic intertwining relations between the unipotent subgroups \( U_{\pm} \) of \( G \) and the pro-unipotent subgroup \( \mathcal{K}_1 \) of \( \mathcal{L}^+ G \). This lemma will be referenced multiple times throughout this work.

\begin{lemma}\label{Lem:IntegLem}
	The following inclusion relations between group subfunctors of \( \mathcal{L} G \) hold:
	\begin{align*}
		\begin{array}{ccc}
			\mu(t) \mathcal{L}^{+} P_+ \mu(t)^{-1} \subseteq \mathcal{L}^{+} G, & \mu(t) \mathcal{L}^{+} U_+ \mu(t)^{-1} \subseteq \mathcal{L}^{1} U_+ \subseteq \mathcal{K}_1,\\
			\mu(t)^{-1} \mathcal{L}^{+} P_- \mu(t) \subseteq \mathcal{L}^{+} G, & \mu(t)^{-1} \mathcal{L}^{+} U_- \mu(t) \subseteq \mathcal{L}^{1} U_- \subseteq \mathcal{K}_1.
		\end{array}
	\end{align*}
\end{lemma}

\begin{proof}
	We will demonstrate the assertions concerning \( P_+ \) and \( U_+ \); the statements for \( P_- \) and \( U_- \) can be shown similarly by changing signs. Let \( R \) be a \( k \)-algebra and set \( A = R[[t]] \). Take an element \( g \in \mathcal{L}^+ P_+(R) = P_+(A) \). By definition of \( P_+ \), the homomorphism of \( A \)-group schemes 
	\[
	f_{\mu, g}: \mathbb{G}_{m,A} \to G_{A}, \quad x \mapsto \mu(x) g \mu(x)^{-1},
	\] 
	extends to a morphism of \( A \)-schemes \( F_{\mu, g}: \mathbb{G}_{a, A} \to G_A  \). Since \( t \in \mathbb{G}_{m} \big(A[\frac{1}{t}]\big) \cap \mathbb{G}_{a}(A) \), we find that 
	\[
	\mu(t) g \mu(t)^{-1} = f_{\mu, g}(t) = F_{\mu, g}(t) \in G(A).
	\]
	If \( g \in \mathcal{L}^+ U_+(R) = U_+(A) \), the functoriality of \( F_{\mu, g} \) for the \( A \)-algebra homomorphism \( \mathrm{pr}: A \to R \) implies that
	\[
	\overline{\mu(t) g \mu(t)^{-1}} = \overline{F_{\mu, g}(t)} = F_{\mu, g}(0) = 1 \in G(R),
	\]
	where ``0" denotes the zero element in \( \mathbb{G}_{a}(R) \). 
\end{proof}

Denote by \( \mathrm{R}_u \mathcal{H}^{\pm} \subseteq \mathcal{H}^{\pm} \) the pullback of \( U_{\pm} \subseteq P_{\pm} \) along \( \mathrm{pr}: \mathcal{L}^+ G \to G \). We have the canonical isomorphisms between group \( k \)-functors:
\[
\mathcal{H}^+ / \mathrm{R}_u \mathcal{H}^+ \cong M \cong \mathcal{H}^- / \mathrm{R}_u \mathcal{H}^- .
\]
 The inclusions \( M \subseteq \mathcal{H}^{\pm} \) induce decompositions of \( \mathcal{H}^{\pm} \) into semidirect products of \( M \) and \( \mathrm{R}_u \mathcal{H}^{\pm} \):
\begin{equation}\label{Eq:semprod}
	\mathcal{H}^+ = M \ltimes \mathrm{R}_u \mathcal{H}^+, \quad \mathcal{H}^- = M \ltimes \mathrm{R}_u \mathcal{H}^-.
\end{equation}
Moreover, the inclusions \( U_{\pm} \subseteq P_{\pm} \subseteq \mathcal{H}^{\pm} \) induce the following decompositions of \( \mathrm{R}_u \mathcal{H}^{\pm} \) and \( \mathcal{H}^{\pm} \):
\[
\mathrm{R}_u \mathcal{H}^{\pm} = U_{\pm} \ltimes \mathcal{K}_1, \quad \mathcal{H}^{\pm} = P_{\pm} \ltimes \mathcal{K}_1.
\]
The next two group subfunctors \( {}^{\pm} \mathcal{H} \) of \( \mathcal{L}^+ G \) play an important role in this work:
\[
{}^+ \mathcal{H} := \mathcal{L}^+ G \cap \mu(t)^{-1} \mathcal{L}^+ G \mu(t), \quad {}^- \mathcal{H} := \mathcal{L}^+ G \cap \mu(t) \mathcal{L}^+ G \mu(t)^{-1}.
\]
Clearly, we have the relation \( {}^- \mathcal{H} = \mu(t) {}^+ \mathcal{H} \mu(t)^{-1} \), or equivalently an isomorphism of group functors:
\[
\mathrm{Ad}_{\mu(t)}: {}^+ \mathcal{H} \cong {}^- \mathcal{H}, \quad g \mapsto \mu(t) g \mu(t)^{-1}.
\]

\begin{definition}[Loop Zip Group]\label{Def:Emuloop}
	We define a group subfunctor \( \mathcal{E}_{\mu}^{\dagger} \subseteq \mathcal{H}^+ \times_k \mathcal{H}^- \) as the pullback in the following diagram:
	\begin{equation*}
		\xymatrix{\mathcal{E}_{\mu}^{\dagger} \ar[rr] \ar[dd] && \mathcal{H}^+ \ar[d]^{\mathrm{pr}}\\
			&& P_+ \ar[d]^{\mathrm{pr}}\\
			\mathcal{H}^- \ar[r] & P_- \ar[r]^{\mathrm{pr}} & M.}
	\end{equation*}
	We call it the \emph{loop zip group} associated with \( \mu \).
\end{definition}

We are now ready to present the key lemma, which plays a fundamental role in this work. We also note that while part (2) of Lemma \ref{Lem:keyInclu} below is useful and its proof is more technical, it is not utilized in the subsequent sections of this work, except in Lemma \ref{Lem:ConnSchbt}. 

\begin{lemma}\label{Lem:keyInclu}
	The following holds:
	\begin{enumerate}
		\item We always have the inclusion \( {}^{\pm} \mathcal{H} \subseteq \mathcal{H}^{\pm} \) of group subfunctors of \( \mathcal{L}^+ G \).
		\item If \( \mu \) is minuscule, we have \( {}^{\pm} \mathcal{H} = \mathcal{H}^{\pm} \); in particular, \( \mathcal{H}^{-} = \mu(t) \mathcal{H}^{+} \mu(t)^{-1} \).
		\item Denote by \( \iota: {}^{\pm} \mathcal{H} \subseteq \mathcal{H}^{\pm} \) the canonical inclusions; then the image of the embedding 
		\[
		(\iota \circ \mathrm{Ad}_{\mu(t)}, \ \iota): {}^{+} \mathcal{H} \to \mathcal{H}^- \times_k \mathcal{H}^+
		\]
		lies in \( \mathcal{E}_{\mu}^{\dagger} \subseteq  \mathcal{H}^- \times_k \mathcal{H}^+ \). In particular, \( \mathcal{E}_{\mu}^{\dagger} \) is mapped to \( \mathsf{E}_{\mu} \) under \( \mathrm{Pr}: \mathcal{H}^- \times_k \mathcal{H}^+ \to P_- \times_k P_+ \).
	\end{enumerate} 
\end{lemma}	

\begin{proof}
	We begin by showing (1) and (3). The proof can be reduced to the case of general linear groups. Choose an embedding of algebraic groups \( G \hookrightarrow \mathrm{GL}_{N,k} \) for some \( N \) and denote by \( \mu' \) the induced cocharacter of \( \mathrm{GL}_{N,k} \). We use \( \mathcal{H}'^{\pm}, {}^{\pm} \mathcal{H}', \mathcal{E}_{\mu'}^{\dagger} \) to denote the counterparts of \( \mathcal{H}^{\pm}, {}^{\pm} \mathcal{H}, \mathcal{E}_{\mu}^{\dagger} \) for \( \mathrm{GL}_{N, k} \). The following commutative diagram holds:
    	\begin{equation}
		\xymatrix{{}^+ \mathcal{H} \ar[r] \ar[d] & \mathcal{L}^+ G \ar[d]\\
			{}^+ \mathcal{H}' \ar[r] & \mathcal{L}^+ \mathrm{GL}_{N, k},}
	\end{equation}
	where all arrows are naturally induced inclusions. From \eqref{Eq:cutParab}, we have
	\[
	\mathcal{H}^{\pm} = \mathcal{H}'^{\pm} \cap \mathcal{L}^+ G, \quad \mathcal{E}_{\mu}^{\dagger} = \mathcal{E}_{\mu'}^{\dagger} \cap \big(\mathcal{H}^- \times_k \mathcal{H}^+\big).
	\]
	From these equalities and the commutative diagram above, it follows that the lemma can be established by dealing with the case where \( G = \mathrm{GL}_{N, k} \) and \( \mu = \mu' \). After possibly conjugating by an element in \( \mathrm{GL}_N(k) \), we may assume that \( \mu \) factors through the standard maximal torus of \( \mathrm{GL}_{N,k} \) consisting of diagonal matrices, given by 
	\[
	t \mapsto \mathrm{diag}(t^{d_1}, \ldots, t^{d_1}, \ldots, t^{d_r}, \ldots, t^{d_r}), \quad d_1 > \ldots > d_r,
	\]
	with multiplicities forming a partition \( (n_1, \ldots, n_r) \) of \( N \).

	Given a \( k \)-algebra \( R \), we may write an element \( g \in G(R) \) in terms of block matrices \( \big(g_{ij}\big)_{1\leq i,j\leq r} \), where each \( g_{ij} \in \mathrm{M}_{n_i \times n_j}(R) \). Such a \( g \) belongs to \( P_+(R) \) if and only if \( g_{ij} = 0 \) for all \( i > j \). An element \( g \in P_+(R) \) belongs to \( U_+(R) \) if \( g_{ii} = \mathrm{I}_{n_i} \), i.e., \( g_{ii} \) is the \( n_i \times n_i \) identity matrix. Similarly, elements in \( P_-(R) \) and \( U_-(R) \) can be characterized by taking \( i < j \).

	Now consider those elements in \( \mathcal{L}^+ G(R) \). By direct computation, for each such element \( g \), we have the equality
	\[
	\mu(t)^{-1} g \mu(t) = \big(t^{d_j - d_i} g_{ij}\big)_{1 \leq i,j \leq r}.
	\]
	Thus, for each \( h \in {}^+ \mathcal{H}(R) \) such that \( h = \mu(t)^{-1} g \mu(t) \) for some \( g \in \mathcal{L}^+ G(R) \), we must have \( \bar{h} \in P_+(R) \), i.e., \( h \in \mathcal{H}^+(R) \),  since for \( i < j \), 
	\[ 
	g_{ij} \in t^{d_i - d_j} \mathrm{M}_{n_i \times n_j}(R[[t]]) \subseteq t \mathrm{M}_{n_i \times n_j}(R[[t]]). 
	\]  
	Hence we also have \( g \in \mathcal{H}^-(R) \). Since \( \bar{h} \) and \( \bar{g} \) share the same Levi component \( \big(\bar{g}_{ii}\big)_{1 \leq i \leq r} \in M(R) \), it follows that \( (g, h) \in \mathcal{E}_{\mu}^{\dagger}(R) \).

	Next, we prove (2). We only demonstrate \( \mathcal{H}^- \subseteq {}^{-} \mathcal{H} \); the equality \( \mathcal{H}^+ \subseteq {}^{+} \mathcal{H} \) can be shown similarly. Since we have \( {}^{-} \mathcal{H} \subseteq \mathcal{H}^- \), we obtain 
	\[
	\mathcal{H}^- \subseteq {}^{-} \mathcal{H} \iff \mu(t)^{-1} \mathcal{H}^- \mu(t) \subseteq \mathcal{L}^+ G \iff \mu(t)^{-1} \mathcal{K}_1 \mu(t) \subseteq \mathcal{L}^+ G,
	\]
	where the second ``\(\iff\)" follows from the decomposition \( \mathcal{H}^- = P_- \ltimes \mathcal{K}_1 \) and the fact \( \mu(t)^{-1} P_- \mu(t) \subseteq \mathcal{K}_1 \) by Lemma \ref{Lem:IntegLem}. Since both \( \mu(t)^{-1} \mathcal{K}_1 \mu(t) \) and \( \mathcal{L}^+ G \) are reduced subschemes of \( \mathcal{L} G \), we may assume \( k \) is algebraically closed, and it suffices to verify the inclusion \( \mu(t)^{-1} \mathcal{K}_1 \mu(t) \subseteq \mathcal{L}^+ G \) at \( k \)-valued points.

	Let \( g \in \mathcal{K}_1 \subseteq G(k[[t]]) \). We need to show \( \mu(t)^{-1} g \mu(t) \subseteq \mathcal{L}^+ G(k) \). By \cite[Exp. XXVI, Theorem 5.1]{SGA3}, as a scheme over \( k \), \( G \) has a Zariski open covering given by 
	\[
	v P_{+} P_{-} = v U_{+} U_{-} M \text{ with } v \in U_-(k).
	\]
	Since the ring \( k[[t]] \) is local, we may assume \( g \) lies in \( (v U_{+} U_{-} M)(k[[t]]) \). Writing 
	\[
	g = vu_+ u_- m, \quad \text{with } u_+ \in \mathcal{L}^+ U_+(k), \, u_- \in \mathcal{L}^+ U_-(k), \, m \in \mathcal{L}^+ M(k),
	\]
	we have \( 1 = \bar{g} = \bar{v} \bar{u}_+ \bar{u}_- \bar{m} \). This leads to \( \bar{v}^{-1} = \bar{u}_+ \bar{u}_- \bar{m} \), yielding \( \bar{u}_+ = \bar{m} = 1 \). Thus, \( u_+ \in \mathcal{L}^+ U_+(k) \cap \mathcal{K}_1(k) \). Using Lemma \ref{Lem:IntegLem} again, we are reduced to showing the inclusion \( \mu(t)^{-1} u_+ \mu(t) \subseteq \mathcal{L}^+ G \).

	Choose a maximal torus \( T \) of \( G \) such that \( \mu \) factors through \( T \subseteq G \). Let \( \Phi \) be the set of roots of \( G \) with respect to \( T \). For each root \( \alpha \in \Phi \), there is an embedding of algebraic groups
	\[
	U_{\alpha}: \mathbb{G}_{a, k} \to G,
	\]
	unique up to an element in \( k^\times \), such that if \( A \) is a \( k \)-algebra, we have
	\[
	t U_{\alpha}(x) t^{-1} = U_\alpha \big(\alpha(t) x\big) \quad \text{ for all } x \in \mathbb{G}_a(A) = A \text{ and all } t \in T(A).
	\]
	Denote also by \( U_\alpha \) for its image in \( G \). We have the following isomorphisms of closed subschemes of \( G \):
	\begin{equation}\label{Eq:RootIsom}
		\prod_{\alpha \in \Phi(\mu)} U_\alpha \xlongrightarrow{\cong} U_+, \quad (u_\alpha)_\alpha \mapsto \prod u_\alpha,
	\end{equation}
	where \( \Phi(\mu) = \{ \alpha \in \Phi \mid \langle \alpha, \mu \rangle = 1 \} \) (since \( \mu \) is minuscule), and the products can be taken in any fixed order. Taking \( A = k((t)) \) and writing \( u_+ = \prod u_\alpha \) according to this decomposition, each \( u_\alpha \) is equal to \( U_\alpha(x_\alpha) \), with \( x_\alpha \in tk[[t]] \subseteq k[[t]] \). Finally, we conclude that for each \( \alpha \in \Phi(\mu) \),
	\[
	\mu(t)^{-1} u_\alpha \mu(t) = U_\alpha \big(x_\alpha t^{\langle \mu^{-1}, \alpha \rangle}\big) = U_\alpha(x_\alpha t^{-1}) \in \mathcal{L}^+ U_+(k).
	\]
\end{proof}
	
\section{Representability of $\prescript{}{1}{\mathcal{C}_1^\mu}$}\label{S:RepCIImu}
In this section, we prove the main result: the representability of Viehmann's double coset space $\prescript{}{1}{\mathcal{C}_1^\mu}$ for $(G, \mu)$ (definition below) in both the equicharacteristic and mixed characteristic cases. We focus primarily on the equicharacteristic case while comparing both versions within this section. 

Again we let \( k, G, \mu \) be as in \S \ref{S:Preliminary}, without restrictions on the characteristic of \( k \). 
Write $\mathrm{Emb}_{\mu}: G \hookrightarrow \mathcal{L}G$ for the embedding $g \mapsto g \mu(t)$. Clearly, the image of $\mathrm{Emb}_{\mu}$ lies in the locus
\[
\mathcal{C}^{\mu} := \mathcal{C}^{\mu(t)} := \mathcal{L}^+ G \mu(t) \mathcal{L}^+ G \subseteq \mathcal{L} G.
\]
Consider the action of $\mathcal{K}_1$ on $\mathcal{L} G$ by right multiplication, and denote by $\mathcal{L} G / \mathcal{K}_1$ the resulting quotient sheaf. Clearly, $\mathcal{C}^{\mu} \subseteq \mathcal{L} G$ is stable under this $\mathcal{K}_1$ action. Denote by 
\[
\mathcal{C}_1^{\mu} := \mathcal{C}^{\mu} / \mathcal{K}_1
\]
the corresponding quotient sheaf. In fact, $\mathcal{C}_1^{\mu}$ is represented by a smooth scheme of finite type over $k$, being a $G$-torsor over the open affine Schubert variety $\mathrm{Gr}_{\mu} := \mathcal{C}^{\mu} / \mathcal{L}^+ G$, which is a smooth $k$-scheme. Similarly, one may consider the right action of $\mathcal{K}_1 \times_k \mathcal{K}_1$ on $\mathcal{L} G$ (resp. on $\mathcal{C}^{\mu}$) given by $g \cdot (k_1, k_2) = k_1^{-1} g k_2$ and the resulting quotient sheaves,
\[
\prescript{}{1}{\mathcal{C}_1^\mu} := \mathcal{K}_1 \backslash \mathcal{C}^{\mu} / \mathcal{K}_1 \subseteq \mathcal{K}_1 \backslash \mathcal{L} G / \mathcal{K}_1.   
\]
In \cite{ViehmannTruncation1}, Viehmann studied certain orbit spaces of (the $\bar{k}$-points of) the double coset space $\mathcal{K}_1 \backslash \mathcal{L} G / \mathcal{K}_1$ in the case where $k$ is algebraically closed of characteristic $p$, where the subspace $\prescript{}{1}{\mathcal{C}_1^\mu}$ comprises building blocks by the Cartan decomposition. We shall refer to the fpqc $k$-sheaf $\prescript{}{1}{\mathcal{C}_1^\mu}$ as \emph{Viehmann's double coset space} of type $\mu$. 

\subsection{The Sheaf $\prescript{}{1}{\mathcal{C}_1^\mu}$ is Representable}\label{S:Rep}

\begin{definition}[The Normal Zip Group of $\mu$]\label{Def:EmuNorm}
	For a cocharacter $\mu: \mathbb{G}_{m,k} \to G$, let $\mathsf{E}_{\mu} \subseteq P_- \times_{k} P_+$ be the closed subgroup whose $A$-valued points for a $k$-algebra $A$ are given by 
	\begin{align*}
		\mathsf{E}_{\mu}(A) = \{(p_- = u_- m, p_+ = u_+ m) \;|\; m \in M(A), \; u_- \in U_-(A), \; u_+ \in U_+(A)\}.
	\end{align*}
\end{definition}

\begin{remark}\label{Rmk:zipgp}
	The zip group $\mathsf{E}_{\mu}$ defined here is a special case of zip groups introduced in \cite{PinkWedhornZiegler1}. The definition in loc. cit. is much more general than those we present here. Later, we will focus on the case where $k$ has characteristic $p$. Then, Frobenius twisted analogues of the zip group defined above will also be considered. For this reason, we call our $\mathsf{E}_{\mu}$ here the \emph{normal zip group} associated with $\mu$, or simply the \emph{zip group} associated with $\mu$ when no confusion arises. 
\end{remark}

Clearly, we have the decomposition $\mathsf{E}_{\mu} = U_- \rtimes M \ltimes U_+$ of algebraic groups over $k$. Since $U_{\pm}$ and $M$ are all connected and smooth over $k$, $\mathsf{E}_{\mu}$ is also connected and smooth. Consider the right action of $\mathsf{E}_{\mu}$ on the $k$-group $G^2 := G \times_{k} G$ given by left multiplication (see our convention in \S \ref{NotatConven}),
\begin{equation*}
	G^2 \times_{k} \mathsf{E}_{\mu} \longrightarrow G^2, \quad \big((g_1, g_2), (p_-, p_+)\big) \longmapsto \big(p_{-}^{-1} g_1, p_+^{-1} g_2\big).    
\end{equation*}
Denote by $\mathsf{E}_{\mu} \backslash G^2$ the corresponding fpqc quotient sheaf. It is a general fact from linear algebraic groups that the quotient sheaf $\mathsf{E}_{\mu} \backslash G^2$ is represented by a scheme that is separated and of finite type over $k$, and the canonical projection map $G^2 \to \mathsf{E}_{\mu} \backslash G^2$ is smooth and surjective; see for example \cite[Proposition 7.15, Theorem 7.18]{MilneAlgebraicGroups}. Hence, it follows from the smoothness of $G$ that $\mathsf{E}_{\mu} \backslash G^2$, viewed as a $k$-scheme, is also smooth (\cite[05B5]{stacks-project}). Moreover, since our $G$ is connected (hence geometrically connected), and the formation of the quotient sheaf $\mathsf{E}_{\mu} \backslash G^2$ commutes with changes of base field, the quotient scheme $\mathsf{E}_{\mu} \backslash G^2$ is also geometrically connected.

Let us consider the following map of fpqc $k$-sheaves:
\begin{equation*}
	\psi_0: G^2 \longrightarrow \prescript{}{1}{\mathcal{C}_1^\mu}, \quad (g, h) \longmapsto g^{-1} \mu(t) h. 
\end{equation*}

\begin{theorem}[The Sheaf $\prescript{}{1}{\mathcal{C}_1^\mu}$ is Representable]\label{Thm:keyRep}The following holds:
	\begin{enumerate}
		\item The map $\psi_0: G^2 \to \prescript{}{1}{\mathcal{C}_1^\mu}$ defined above is $\mathsf{E}_{\mu}$-invariant and hence induces a map of fpqc $k$-sheaves,
		\[
		\psi: \mathsf{E}_{\mu} \backslash G^2 \longrightarrow \prescript{}{1}{\mathcal{C}_1^\mu}.
		\]
		
		\item The induced map $\psi$ is an isomorphism of $k$-sheaves. In particular, $\prescript{}{1}{\mathcal{C}_1^\mu}$ is represented by a geometrically connected scheme, smooth and separated of finite type over $k$. Moreover, we have $\dim \prescript{}{1}{\mathcal{C}_1^\mu} = \dim G$. 
		
		\item \label{Thm:keyRepMulp} Let $n\geq 1$ be an integer and set $\mu' = n\mu$. Then the association $g \mu(t) h \mapsto g \mu'(t) h$ on local sections gives a well-defined isomorphism of $k$-schemes:
		\[
		\epsilon: \prescript{}{1}{\mathcal{C}_1^\mu} \cong \prescript{}{1}{\mathcal{C}_1^{\mu'}},
		\]
		fitting into the commutative diagram below:
		\begin{equation*}\label{qIdentity}
			\xymatrix{\mathsf{E}_{\mu} \backslash G^2 \ar@{=}[d] \ar[r]^{\cong} & \prescript{}{1}{\mathcal{C}_1^\mu} \ar[d]_{\cong}^{\epsilon} \\
				\mathsf{E}_{\mu'} \backslash G^2 \ar[r]^{\cong} & \prescript{}{1}{\mathcal{C}_1^{\mu'}}.}
		\end{equation*}
	\end{enumerate}
\end{theorem}

\begin{proof}
	(1) Let $A$ be a $k$-algebra. For all $g, h \in G(A)$ and $(p_-, p_+) \in \mathsf{E}_{\mu}(A)$, we have the equality
	\[
	g^{-1} p_- \mu(t) p_+^{-1} h = X g^{-1} \mu(t) h Y,
	\]
	where we define
	\[
	X := g^{-1} \big(\mu(t) u_+^{-1} \mu(t)^{-1}\big)g, \quad Y := (u_+^{-1} h)^{-1} \big(\mu(t)^{-1} u_- \mu(t)\big)(u_+^{-1} h).
	\]
	Here, $u_+$ and $u_-$ are the unipotent radical components of $p_+$ and $p_-$, respectively. From Lemma \ref{Lem:IntegLem}, we know that both $X$ and $Y$ lie in $\mathcal{K}_1(A)$. 

	(2) The map $\psi$ is clearly surjective, and so we only need to show that it is also injective as a map of presheaves. For this, it suffices to show that the composition below is injective:
	\[
	\big(\mathsf{E}_{\mu} \backslash G^2\big)^{\mathrm{pre}} \hookrightarrow \mathsf{E}_{\mu} \backslash G^2 \to \prescript{}{1}{\mathcal{C}_1^\mu},
	\]
	where $\big(\mathsf{E}_{\mu} \backslash G^2\big)^{\mathrm{pre}}$ denotes the obvious presheaf whose fpqc sheafification is $\mathsf{E}_{\mu} \backslash G^2$. Take $g_1, h_1, g_2, h_2 \in G(A)$ and suppose that there exist $k_1, k_2 \in \mathcal{K}_1(A)$ such that 
	\[
	g_1^{-1} \mu(t) h_1 = k_1^{-1} g_2^{-1} \mu(t) h_2 k_2.
	\]
	Write $g := g_2 k_1 g_1^{-1}$ and $h := h_2 k_2 h_1^{-1}$. Then by the definitions of ${}^+\mathcal{H}$ and ${}^-\mathcal{H}$, we have 
	\[
	h \in {}^+\mathcal{H}(A), \quad g = \mu(t) h \mu(t)^{-1} \in {}^-\mathcal{H}(A), \text{ and } (g, h)\in \mathcal{E}^{\dagger}(A).
	\]
	By applying Lemma \ref{Lem:keyInclu}, we know $(\bar{g}, \bar{h}) \in \mathsf{E}_{\mu}(A)$. Therefore, we have
	\[
	(g_1, h_1) = (g_2, h_2) \cdot (\bar{g}^{-1}, \bar{h}^{-1}).
	\]
	Hence, $\psi$ is an isomorphism. For the dimension of $\prescript{}{1}{\mathcal{C}_1^\mu}$, we note that the multiplication map 
	\[
	U_- \times_k M \times_k U_+ \to G,
	\]
	is an open immersion of algebraic varieties over $k$; see for example \cite[Theorem 13.33]{MilneAlgebraicGroups}. It follows from this fact that $\dim \mathsf{E}_{\mu} = \dim G$, and hence we have
	\[
	\dim \prescript{}{1}{\mathcal{C}_1^\mu} = 2 \dim G - \dim \mathsf{E}_{\mu} = \dim G. 
	\] 
	
	(3) We note that $\mu$ and $\mu'$ determine the same algebraic groups $M \subseteq P_{\pm}$ (see e.g., \cite[Theorem 4.1.7]{ConradReductiveGroupSchemes}) and hence the same zip groups $\mathsf{E}_{\mu} = \mathsf{E}_{\mu'}$. Thus, (3) essentially follows from (2). 
\end{proof}

\begin{example} The reader may find the next example amusing, where our cocharacter $\mu$ is trivial:
\[
G \cancong \mathcal{K}_1 \backslash \mathcal{L}^+ G \cancong \mathcal{K}_1 \backslash \mathcal{L}^+ G / \mathcal{K}_1 \cancong \mathcal{L}^+ G / \mathcal{K}_1 \cancong G.
\]
\end{example}

\begin{remark}\label{Rmk:trunSch}
	Closely related to the sheaf $\prescript{}{1}{\mathcal{C}_1^\mu}$ is the quotient stack $[\mathcal{K}_1 \backslash \mathcal{C}^\mu / \mathcal{K}_1]$, referred to as the \emph{moduli stack of 1-1 truncated modifications of type $\mu$}. For any $k$-algebra $R$, the groupoid $[\mathcal{K}_1 \backslash \mathcal{C}^\mu / \mathcal{K}_1](R)$ is equivalent to the groupoid of triples \[(\mathcal{E} \overset{\alpha}{\dashrightarrow} \mathcal{F}, \; \mathcal{E}_0|_{\mathrm{Spec} R} \overset{\epsilon_1}{\cong} \mathcal{E}|_{\mathrm{Spec} R}, \; \mathcal{E}_0|_{\mathrm{Spec} R} \overset{\epsilon_2}{\cong} \mathcal{F}|_{\mathrm{Spec} R}),\]
	where $\mathcal{E}, \mathcal{F}$ are $G$-torsors over $\mathrm{D}_R := \mathrm{Spec} R[[t]]$, $\mathcal{E}_0$ is the trivial $G$-torsor over $\mathrm{D}_R$, and $\alpha: \mathcal{E} \overset{\alpha}{\dashrightarrow} \mathcal{F}$ is a modification of type $\mu$. In other words, it is a morphism of $G$-torsors $\mathcal{E}|_{\mathrm{D}_R^*} \overset{\alpha}{\rightarrow} \mathcal{F}|_{\mathrm{D}_R^*}$ of type $\mu$, where $\mathrm{D}_R^* := \mathrm{Spec} R((t))$. The term ``1-1 \emph{truncation}" indicates that the trivializations $\epsilon_1$ and $\epsilon_2$ are confined to the closed subscheme $\mathrm{Spec} R = \mathrm{Spec} R[[t]] / t^n$ of $\mathrm{D}_R$, with $n=1$.

	Since $\prescript{}{1}{\mathcal{C}_1^\mu}$ equals the coarse moduli sheaf of $[\mathcal{K}_1 \backslash \mathcal{C}^\mu / \mathcal{K}_1]$, our theorem regarding the representability of $\prescript{}{1}{\mathcal{C}_1^\mu}$ can also be expressed as follows: \emph{the coarse moduli space of $[\mathcal{K}_1 \backslash \mathcal{C}^\mu / \mathcal{K}_1]$ exists as a scheme.}
\end{remark}

\subsection{Connections with affine Schubert Varieties}

It follows from the relation $\mu(t)^{-1} \mathcal{L}^+ P_- \mu(t) \subseteq \mathcal{L}^+ G$ in Lemma \ref{Lem:IntegLem} and the relation ${}^-\mathcal{H} \subseteq \mathcal{H}^{-}$ in Lemma \ref{Lem:keyInclu} that there is always an immersion $j: P_- \backslash G \to \mathrm{Gr}_{\mu}$ from the (partial) flag variety $P_- \backslash G$ to the open affine Schubert variety $\mathrm{Gr}_{\mu} := \mathcal{L}^+G \mu(t) \mathcal{L}^+G / \mathcal{L}^+G$, induced by the association $g \mapsto g^{-1} \mu(t)$. Consider the following diagram of $k$-schemes:
\begin{equation}\label{eq:noncommdiag}
	\xymatrix{\mathcal{C}_1^{\mu} \ar[d] \ar[r] & \prescript{}{1}{\mathcal{C}_1^\mu} \ar[r]^{\cong \ \ \ } & \mathsf{E}_{\mu} \backslash G^2 \ar[d] \\
		\mathrm{Gr}_{\mu} && P_{-} \backslash G \ar@{_{(}->}_j[ll],}
\end{equation}
where the left vertical arrow is the canonical projection and the right vertical arrow is induced by the composition of the first projection $G^2 \xrightarrow{\mathrm{pr}_1} G$ with the canonical projection $G \to P_- \backslash G$. From the surjectivity of the left vertical arrow, we see that this diagram cannot be commutative unless $j$ is an isomorphism:

\begin{lemma}\label{Lem:ConnSchbt}
	If $\mu$ is minuscule, the projection $\mathcal{C}_1^\mu \twoheadrightarrow \mathrm{Gr}_{\mu}$ factors through the quotient $\mathcal{C}_1^\mu \twoheadrightarrow \prescript{}{1}{\mathcal{C}_1^\mu}$.
\end{lemma}
\begin{proof}
	In case $\mu$ is minuscule, it is a well-known fact that $j$ becomes an isomorphism. This fact can be demonstrated using Lemma \ref{Lem:keyInclu}. 
\end{proof}

\subsection{Connection with Lusztig's Variety}\label{S:Lusztig}
After the discovery of the representability result \ref{Thm:keyRep}, we learned from the paper of Pink-Wedhorn-Ziegler \cite{PinkWedhornZiegler1} that the representing quotient scheme $\mathsf{E}_{\mu} \backslash G^2$ has several other incarnations. To be precise, $\mathsf{E}_{\mu} \backslash G^2$ also represents some of the coset varieties defined in loc. cit., which, in turn, is isomorphic to special cases of Lusztig's varieties defined in \cite{LuszParobCharShvII}. For the reader's convenience, we detail below the corresponding coset variety and Lusztig variety, following \cite[\S 12.1]{PinkWedhornZiegler1}, with slightly modified normalization (e.g., left actions vs. right actions). 

We assume in this subsection that $k$ is algebraically closed, as in the setting of \cite{PinkWedhornZiegler1} and \cite{LuszParobCharShvII}. Thanks to this assumption, for a $k$-variety $X$, we abuse language and do not distinguish between $X$ and $X(k)$. 

Via the following isomorphism $\iota: G^2\cong G^2, (g, h)\mapsto (g^{-1}, h^{-1})$, 
	the $\mathsf{E}_{\mu}$ action on  $G^2$ that we define in  \S \ref{S:Rep}  translates into the usual action by right multiplication: $(g, h)\circ (p_-, p_+)= (gp_-, hp_+)$.  Denote by $ G^2/\mathsf{E}_{\mu}$  the quotient scheme of $G^2$ by the action of $\mathsf{E}_{\mu}$ by the usual right multiplication. Then  $\iota$ induces an isomorphism of algebraic $k$-varieties, \[\iota: \mathsf{E}_{\mu}\backslash G^2 \cong G^2/\mathsf{E}_{\mu}.\]
	
	Note that the tuple $ Z:=(G, P_-, P_+, \mathrm{id}: M\to M)$ forms a connected algebraic datum in the sense of \cite[Definition3.1]{PinkWedhornZiegler1}. 
	\begin{definition}[cf. {\cite[Definition12.3]{PinkWedhornZiegler1}}]
		The coset space associated to the connected zip datum $Z$ is the set $C_{Z}$ of all triples $(X, Y, \Phi)$ consisting of a right $P_-$-coset $X\subseteq G$, a right $P_+$-coset $Y\subseteq G$, and an isomorphism $\Phi: U_-\backslash X \to U_+\backslash Y$ of $M$-torsors. Here $M$ acts on $U_-\backslash X, U_+\backslash Y$ by right multiplication.
	\end{definition}
	
	A standard base point of $C_Z$ is given by the triple $(P_-, P_+, \mathrm{id}_M)$.  Given $g, h\in G$, denote by
	\[l_g: U_-\backslash P_-\cong U_-\backslash P_-g,\ \ \  l_h: U_+\backslash P_+\cong U_+\backslash P_+ h\]
	the isomorphisms of $M$-torsors induced by right multiplication by $g, h$ respectively.  
	
	\begin{lemma}[cf. {\cite[Lemma 12.5]{PinkWedhornZiegler1}}]
		There is a unique structure of algebraic variety on $C_{Z}$ such that 
		\[G^2\to C_{Z}, \ \ (g, h)\mapsto (P_-g, \ P_+h, \ l_h  \ \circ l_g^{-1} ),\]
		induces an isomorphism of algebraic varieties, $G^2/\mathsf{E}_{\mu}\cong C_{Z}$.  
	\end{lemma}
	
	We now present the corresponding Lusztig variety, for which we need to fix some group theoretic data. Let $(B, T)$ be a 
    Borel pair of $G$. Let $W$ be the Weyl group of $G$ with respect to $T$ and $I\subseteq W$ the set of simple reflections determined by $B$. Let $J, K\subseteq I$ be the type of parabolic subgroups $P_-, P_+$ respectively. Since $P_-$ is the opposite parabolic of $P_+$, we have $K={}^{w_0}J:=w_0Jw_0^{-1}=w_0Jw_0$. 
	For a subset $L\subseteq I$, we denote by $W_{L}\subseteq W$ the subgroup of $W$ generated by elements in $L$, and ${}^LW\subseteq W$ (resp. $W^{L}\subseteq W$ ) the set of minimal representatives of the coset space $W_L\backslash W$ (resp. $W/W_L$). 
	Then ${}^KW^J:={}^KW\cap W^J$ forms the set of minimal representatives of the double coset space $W_K\backslash W/ W_J$. Let $x\in {}^KW^J$ be the unique element of maximal length in ${}^KW^J$ (it is also the unique element of minimal length in the double coset $W_Kw_0W_J$, with $w_0\in W$ the unique element of maximal length), then we have 
	\[x=w_{0, K}w_0=w_0w_{0,J},\]
	where $w_{0, K}\in W_K, w_{0, J}\in W_J$ are the unique elements of maximal length.  We then also have the relation that $K={}^xJ$. We refer to \cite[\S 2.2]{PinkWedhornZiegler1} and the reference therein for more detailed discussion on the combinatorial data involved here. 
	
	The Lusztig variety $Z_{J, \mathrm{id}, x}$ is the algebraic $k$-variety parametrizing triples $(P, Q, [k])$ consisting of parabolic subgroups  $P, Q$ of $G$ of types $J,K$ respectively, and a coset $[g]\in U_-\backslash G / U_+ $, such that 
	\[\mathrm{relpos}(Q, {}^gP)=x. \]
	Here $\mathrm{relpos}(Q, P^g)$ denotes the relative position of $Q$ and $P^g:=g^{-1}Pg$; see for example \cite[\S 2]{Moonen&WedhornDiscreteinvariants} for more details. A standard base point of $Z_{J, \mathrm{id}, x}$ is given by the triple $(P_-, P_+, [1])$ by Lemma 2.9 of loc. cit.
	
	\begin{proposition}[{\cite[12.13]{PinkWedhornZiegler1}}]	
		The following is an isomorphism of $k$-varieties, 
		\[C_{Z}\cong Z_{J, \mathrm{id}, x},  \ \ \ (P_-g, P_+h, l_h\circ l_g^{-1} )\mapsto ({}^gP_-, {}^hP_+, [hg^{-1}])\]
	\end{proposition}
	
	To summarize, we have a sequence of isomorphisms between algebraic $k$-varieties, 
	\[\prescript{}{1}{\mathcal{C}_1^\mu} \xleftarrow[\psi]{\cong}\mathsf{E}_{\mu}\backslash G^2\cong G^2/\mathsf{E}_{\mu}\cong C_{Z}\cong Z_{J, \mathrm{id}, x}.\]
	
	Note that the last three isomorphisms are essentially formal. These isomorphisms provide different incarnations of our fpqc sheaf $\prescript{}{1}{\mathcal{C}_1^\mu}$. In particular, $\prescript{}{1}{\mathcal{C}_1^\mu}$ is also represented by the Lusztig variety $Z_{J, \mathrm{id}, x}$.

	\subsection{Perfect representability of   $\prescript{}{1}{\mathrm{C}_1^\mu}$.}\label{S:MixRep}
	In this subsection, we let $k$ be a perfect field of characteristic $p$. Consider the pair $(\mathcal{G}, \mu)$ consisting of a reductive group $\mathcal{G}$ over $W(k)$, and a cocharacter $ \mu: \mathbb{G}_{m, W(k)}\to \mathcal{G}$ defined over $W(k)$. Write $G=\mathcal{G}_k$ and still use
	$\mu$ to denote its base change to $k$.
	\begin{definition} Let $\mathrm{Alg}_k$ and $\mathrm{Perf}_k$ denote the category of  $k$-algebras and perfect $k$-algebras respectively. 
		We say a presheaf $\mathcal{F}$ on $\mathrm{Perf}_k$  is \emph{perfectly representable} if there exists $k$-scheme $X$ such that its restriction on $\mathrm{Perf}_k$ (viewed as a presheaf on $\mathrm{Alg}_k$) is isomorphic to $\mathcal{F}$. 
	\end{definition}

	Denote by $K:=L^+\mathcal{G}$ the mixed-characteristic positive loop group of $\mathcal{G}$, which is a perfect group scheme over $k$ (meaning, a group functor over $\mathrm{Perf}_k$) given by associating a perfect $k$-algebra $R$ with the group $K(R):=\mathcal{G}(W(R))$. Similarly, we can define the mixed-characteristic loop group $L\mathcal{G}, R\mapsto \mathcal{G}(W(R)\frac{1}{p})$ and the corresponding $K_1$ given by the kernel of the natural reduction modulo $p$ map $K\to G$. 
	Consider the mixed characteristic version of $\prescript{}{1}{\mathcal{C}_1^\mu}$,
	\[\prescript{}{1}{\mathrm{C}_1^\mu}:=K_1\backslash K\mu(p)K/K_1.\]
	
	\begin{theorem}\label{Thm:eqvsmix}
		The sheaf $\prescript{}{1}{\mathrm{C}_1^\mu}$ on $\mathrm{Perf}_k$ is perfectly represented by $\prescript{}{1}{\mathcal{C}_1^\mu}$. In other words, we have,
		\[\prescript{}{1}{\mathrm{C}_1^\mu}=(\prescript{}{1}{\mathcal{C}_1^{\mu}) ^\mathrm{perf}}.\]
	\end{theorem}
	
	\begin{proof}
		First note that since $W(R)$ is $p$-adically complete, the smoothness of $\mathcal{G}$  implies that for each perfect $k$-algebra $R$, every point $g\in G(R)$ admits a lift  $\tilde{g}\in \mathcal{G}(W(R))$ and all the lifts of $g$ form the coset $K_1(R)g=gK_1(R)$.  
		Consider the following map of presheaves on $\mathrm{Perf}_k$,
		\begin{equation*}
			(G^2)^\mathrm{Perf} \to \prescript{}{1}{\mathrm{C}_1^{\mu}}, \ \ \ (g, h)\longmapsto K_1\tilde{g}^{-1}\mu(p)\tilde{h} K_1. 
		\end{equation*}
		This map is clearly surjective and we need to show the injectivity. For this, one verifies easily that mixed-characteristic versions of Lemma \ref{Lem:IntegLem} and Lemma \ref{Lem:keyInclu} still hold: the proof works verbatim, with necessary but harmless modification; c.f., \cite[Lemma 3.4.1]{Yan24}. For example, in the mixed setting, the semidirect products in \eqref{Eq:semprod} can be replaced by corresponding short exact sequences. 
	\end{proof}
	
	\begin{remark}\label{Rmk:EqvsMix} Let us assume that $\mathcal{G}$ is defined over $\mathbb{Z}_p$. Choose a Borel pair $(\mathcal{B}, \mathcal{T})$ of $\mathcal{G}$, both defined over $\mathbb{Z}_p$ (this is possible). Assume further that $k$ is algebraically closed and write $(B, T)=(\mathcal{B}, \mathcal{T})\otimes k$. Then we have the identification \[W_{\mathcal{G}}=\mathrm{N}_{\mathcal{G}}(\mathcal{T})(W(k))/\mathcal{T}(W(k))=\mathrm{N}_{G}(T)(k)/T(k)=W_G\] of Weyl groups and the identification $J(\mu)=J(\mu_k)$ of types for $\mu$ and $\mu_k$ respectively,  and hence the identification ${}^{J(\mu)}W_{\mathcal{G}}={}^{J(\mu_k)}W_G$, which we abbreviate as ${}^JW$, of minimal length representatives of $W_\mu\backslash W$ (c.f., \S \ref{S:Lusztig}).
		
		Recall that in \cite{ViehmannTruncation1}, Viehmann proves the following two homeomorphisms 
		\begin{equation}\label{Vieheq}
			\sigma^{-1}({}^JW)\cong \prescript{}{1}{\mathcal{C}_1^\mu}(k)/\mathrm{Ad}_\sigma \mathcal{K}(k)=\prescript{}{1}{\mathcal{C}_1^\mu}(k)/\mathrm{Ad}_\sigma G(k), \quad  w\mapsto w w_0 w_{0, \mu}\mu(t),
		\end{equation} and 
		\begin{equation}\label{Viehmix}
			\sigma^{-1}({}^JW)\cong \prescript{}{1}{\mathrm{C}_1^\mu}(k)/\mathrm{Ad}_\sigma K(k)=\prescript{}{1}{\mathrm{C}_1^\mu}(k)/\mathrm{Ad}_\sigma G(k), \quad w\mapsto \widetilde{w w_0 w_{0, \mu}}\mu(p),
		\end{equation} simultaneously by running parallel arguments. Here we abuse notation and use the same notation for any representative in $\mathrm{N}_{G}(T)(k)$ of $ww_0w_{0, \mu}$  and use $\widetilde{w w_0 w_{0, \mu}}$ for any of its representative in $\mathrm{N}_{\mathcal{G}}(\mathcal{T})(W(k))$, where $w_0\in W, w_{0, \mu}\in W_\mu $ are the longest elements.  The set of minimal length representatives ${}^JW$ is equipped with the topology  \footnote{See for example \S 2A of \cite{PinkWedhornZiegler2} for a discussion on the topology of a finite $T_0$ space.}  induced by the partial order $\preccurlyeq$ defined as in Definition 6.1 of \cite{PinkWedhornZiegler1},  which generalizes the partial order defined by X. He in \cite{HeGstabwond}. Precisely, a subset $U\subseteq {}^JW$ is open if for any $w \in {}^JW$ and any $v\in W$, the condition $w\preccurlyeq v$ implies $v\in U$. This topology of ${}^JW$ induces that of $\sigma^{-1}({}^JW$) on the left hand sides of \eqref{Vieheq} and \eqref{Viehmix}. 
        
		The common ${}^JW$ in the two homeomorphisms above simply look like a nice coincidence, as the RHS of \eqref{Vieheq} and \eqref{Viehmix} were not directly related.  Now with Theorem \ref{Thm:eqvsmix}, this can be explained as:
		\emph{the right hand sides of \eqref{Vieheq} and of \eqref{Viehmix} are already identified  before taking $G(k)$-$\sigma$ conjugacy classes}.
	\end{remark}

    \subsection{An Open Problem} \label{S:Conjecture}

Suppose now that $G$ is split over $k$ (without constraint on the characteristic of $k$) and that we have chosen a Borel pair $(B, T)$ with a maximal torus $T$ that is split. We then have the Cartan decomposition of $\mathcal{L}G$ into Cartan cells $\mathcal{C}^\mu$, which are locally closed subschemes:
\[
\mathcal{L}G = \bigsqcup_{\mu \in X_*(T)_{\mathrm{dom}}}\mathcal{C}^\mu.
\]
The closure of $\mathcal{C}^\mu$, denoted by $\mathcal{C}^{\leq \mu}$, is (set-theoretically) equal to the union of $\mathcal{C}^{\mu'}$ with $\mu' \leq \mu$. Here, $\mu' \leq \mu$ if $\mu - \mu'$ is equal to a nonnegative integral linear combination of positive coroots. Given the representability of $\prescript{}{1}{\mathcal{C}_1^\mu}$, a natural question arises as to whether the fpqc $k$-sheaf,
\[
\prescript{}{1}{\mathcal{C}_1^{\leq \mu}} := \mathcal{K}_1 \backslash \mathcal{C}^{\leq \mu} / \mathcal{K}_1,
\]
is also represented by a $k$-scheme.

\begin{conjecture}
	The fpqc sheaf $\prescript{}{1}{\mathcal{C}_1^{\leq \mu}}$ is represented by a normal $k$-scheme. 
\end{conjecture}

This conjecture is true for minuscule $\mu$, as in that case we have $\prescript{}{1}{\mathcal{C}_1^\mu} = \prescript{}{1}{\mathcal{C}_1^{\leq \mu}}$. We currently have no insight on how to prove even a single non-trivial case. The validity of this conjecture would imply that the larger space $\mathcal{K}_1 \backslash \mathcal{L}G / \mathcal{K}_1$ is an ind-scheme. 

\section{More on the Variety \( \prescript{}{1}{\mathcal{C}_1^\mu} \)}
\label{S:Properties} This section collects various results related to the variety \( \prescript{}{1}{\mathcal{C}_1^\mu} \). Some are consequences or corollaries of the previous section; others, while not needed later, are included for completeness. Our presentation of Proposition~\ref{Prop:qaffine} is deliberately expository.

	\subsection{The algebraicity of quotient stack $[\prescript{}{1}{\mathcal{C}_1^\mu}/\mathrm{Ad}_{\tau}G]$}\label{S:QuotofCIImu}
    Let $\tau: \mathcal{L}^+ G \to \mathcal{L}^+ G$ be an endomorphism of $k$-group schemes. Consider the $\mathcal{L}^+ G$-$\tau$ conjugation action on $\prescript{}{1}{\mathcal{C}_1^\mu}$ (viewed as an fpqc sheaf)
	\begin{equation*}
		h\cdot g:=g^{-1}h\tau(g), \ \  g\in \mathcal{L}^+ G,\ \  h\in \prescript{}{1}{\mathcal{C}_1^\mu}.  
	\end{equation*}
	Suppose that we have $\tau(\mathcal{K}_1)\subseteq \mathcal{K}_1$ so that the $\mathcal{L}^+ G$-action factors through the quotient $\mathrm{pr}:\mathcal{L}^+ G\to G$; this is the case if $\tau$ is induced by some $k$-endomorphism of $G$. We denote the resulting quotient $k$-stack by
	\[ [\prescript{}{1}{\mathcal{C}_1^\mu}/\mathrm{Ad}_{\tau}G]. \]
	Note that one can form such a quotient stack for the sheaf $\prescript{}{1}{\mathcal{C}_1^\mu}$ (due to our convention concerning stacks in \S \ref{NotatConven}, here we view $\prescript{}{1}{\mathcal{C}_1^\mu}$ as a sheaf over the site $(\mathrm{Aff}/k)_{\mathrm{et}}$), without knowing whether it is representable or not.  An immediate consequence of  Theorem \ref{Thm:keyRep} is that the stack $[\prescript{}{1}{\mathcal{C}_1^\mu}/\mathrm{Ad}_{\tau}G]$ is an \emph{  algebraic} stack. In fact, more is true:
	
	\begin{corollary}\label{Cor:ArtStack}
		The $k$-stack $[\prescript{}{1}{\mathcal{C}_1^\mu}/\mathrm{Ad}_{\tau} G]$ is a smooth Artin $k$-stack of dimension $0$.
	\end{corollary}
	\begin{proof}
		It is an Artin stack because by Theorem \ref{Thm:keyRep}, the sheaf $\prescript{}{1}{\mathcal{C}_1^\mu}$ represented by an algebraic $k$-scheme and the canonical projection $\prescript{}{1}{\mathcal{C}_1^\mu}\to [\prescript{}{1}{\mathcal{C}_1^\mu}/\mathrm{Ad}_{\tau}G]$ is smooth (since $G$ is $k$-smooth). It is smooth over $k$ because $\prescript{}{1}{\mathcal{C}_1^\mu}$ is smooth over $k$. Its dimension is computed by: \[\dim[\prescript{}{1}{\mathcal{C}_1^\mu}/\mathrm{Ad}_{\tau} G]=\dim \prescript{}{1}{\mathcal{C}_1^\mu}-\dim G=0.\]
	\end{proof}

	\subsection{Some closed subschemes of $\prescript{}{1}{\mathcal{C}_1^\mu}$}We consider in this subsection embeddings of $G$ into $\mathcal{L} G$ induced by the cochracter $\mu: \mathbb{G}_{m, k}\to G$ in various ways. We will show that both $G/U_-$ and $U_+\backslash G$ embed into $\prescript{}{1}{\mathcal{C}_1^\mu}$ as closed subschemes. Such embeddings will reappear in \S \ref{S:ShtukAsZip} and \S \ref{S:App}.
	
	Recall that we have in \S \ref{S:RepCIImu} introduced the embedding $ \mathrm{Emb}_{\mu}:  G\hookrightarrow\mathcal{L}G, g\mapsto g\mu(t)$. Here we consider another embedding of $G$ into $\mathcal{L} G$,
	\[\mathrm{Emb}_{\mathrm{ad}(\mu)}: G\hookrightarrow \mathcal{L} G, \ \ \ g\mapsto \mu(t)^{-1}g\mu(t).\]
	Clearly this is in fact a \emph{homomorphism} of group functors. By Lemma \ref{Lem:IntegLem}, $\mathrm{Emb}_{\mathrm{ad}(\mu)}$ induces an embedding of group schemes \(U_-\to \mathcal{K}_1\). Dually we can consider the embeddings
	\[\prescript{}{\mu}{\mathrm{Emb}}: G\hookrightarrow \mathcal{L} G, \ g\mapsto \mu(t)g, \quad \mathrm{Emb}_{\mathrm{ad}(\mu^{-1})}:G\hookrightarrow \mathcal{L} G, g\mapsto \mu(t)g\mu(t)^{-1}.\]
	Again by Lemma \ref{Lem:IntegLem}, $\mathrm{Emb}_{\mathrm{ad}(\mu^{-1})}$ induces an embedding $U_+\hookrightarrow \mathcal{K}_1$. For our later applications we are interested in the relation between the quotient scheme $G/U_-$ resp. $G/U_+$ of $G$ and the quotient scheme $\mathcal{C}_1^{\mu}$ resp. $\prescript{}{1}{\mathcal{C}^\mu}:=\mathcal{K}_1\backslash \mathcal{C}^{\mu}$ of $\mathcal{C}^{\mu}$.  We denote by 
	\[ G\mu(t)\mathcal{K}_1/\mathcal{K}_1\subseteq \mathcal{C}_1^{\mu}, \ \ \ \text{resp.} \ \ \mathcal{K}_1\backslash \mathcal{K}_1\mu(t)G\subseteq \prescript{}{1}{\mathcal{C}^\mu}\] 
	the fpqc sheafification of the obvious presheaves under consideration.
	Clearly $G\mu(t)\mathcal{K}_1/\mathcal{K}_1$ is the $ G $-orbit of (the image of) $ \mu(t) $ inside $\mathcal{C}_1^{\mu}$; here $ G $ acts  on $ \mathcal{C}_1^{\mu} $  by left multiplication. Hence it follows from the general theory of smooth algebraic group action on schemes that $ G\mu(t)\mathcal{K}_1/\mathcal{K}_1$ is represented by a scheme, smooth separated of finite type over $ k $, and the embedding $G\mu(t)\mathcal{K}_1/\mathcal{K}_1 \hookrightarrow \mathcal{C}_1^{\mu}$ is an immersion. An analogous discussion holds for $\mathcal{K}_1\backslash \mathcal{K}_1\mu(t)G$ and $\prescript{}{1}{\mathcal{C}^\mu}$.
	\begin{corollary} \label{Cor:Embed} The following holds:
		\begin{enumerate}
			\item The embedding $\mathrm{Emb}_{\mu}: G\to \mathcal{C}^{\mu}  $ induces an immersion of  $k$-schemes, 
			\begin{align*}
				\tilde{\alpha}: G/U_-\cong G\mu(t)\mathcal{K}_1/\mathcal{K}_1 \subseteq \mathcal{C}^{\mu}_1.
			\end{align*} 
			\item 	The composition $ \alpha: G/U_{-}\to \mathcal{C}_1^{\mu} \xrightarrow{\mathrm{pr}} \prescript{}{1}{\mathcal{C}_1^\mu}$ is a closed immersion of $k$-schemes. Hence the canonical projection $G\mu(t)\mathcal{K}_1/\mathcal{K}_1 \to \prescript{}{1}{\mathcal{C}_1^\mu}$ is a closed immersion. 
			
			\item The embedding $\prescript{}{\mu}{\mathrm{Emb}}: G\to \mathcal{C}^\mu $ induces an immersion of $k$-schemes, 
			\begin{align*}
				\tilde{\beta}:	U_+\backslash G \cong \mathcal{K}_1\backslash\mathcal{K}_1 \mu(t)G\subseteq  \prescript{}{1}{\mathcal{C}^\mu}.
			\end{align*} 
			
			\item 	The composition $ \beta: U_+\backslash G\to \prescript{}{1}{\mathcal{C}^\mu} \xrightarrow{\mathrm{pr}} \prescript{}{1}{\mathcal{C}_1^\mu}$ is a closed immersion of $k$-schemes. In particular, the canonical projection $\mathcal{K}_1\backslash\mathcal{K}_1 \mu(t)G \to \prescript{}{1}{\mathcal{C}_1^\mu}$ is a closed immersion.

		\end{enumerate}
	\end{corollary}	
	\begin{proof} Below we only show (1) and (2). The proof of (3) and (4) is symmetric. 
		
		(2)  Using the isomorphism $ \psi:\mathsf{E}_{\mu}\backslash G^2 \cong \prescript{}{1}{\mathcal{C}_1^\mu}$, we obtain by transport of structure a morphism,
		\[\alpha: G/U_- \to \mathsf{E}_{\mu}\backslash G^2, \ \ \ g\mapsto (g^{-1}, 1). \]
		It is easy to see that we have pull-back diagram below, where $ \pi_1 $ is given by the composition of the  first projection $ \mathrm{pr}_1:G^2 \to G $ with the canonical projection $ G\to G/U_- $,  and $ d: G\times_k P_+\to G^2 $ is given by $ (g, p_+=u_+m)\longmapsto (mg^{-1}, p_+) $. 
		\begin{equation*}
			\xymatrix{G\times_{k}P_+\ar[r]^{d}\ar[d]_{\pi_1}& G^{2}\ar[d]^{\mathrm{pr}}\\
				G/U_-\ar[r]^{ \alpha}&\mathsf{E}_{\mu}\backslash G^2.  }
		\end{equation*}
		Clearly $ d $ is a monomorphism with image $ G\times_{k}P_+\subseteq G^{2}$, and hence is a closed immersion, since every subgroup of an algebraic group over $k$ is a closed subgroup; see e.g., \cite{MilneAlgebraicGroups}. Finally we conclude by faithfully flat descent ($ \mathrm{pr} $ is faithfully flat). 
		
		(1)  Since the embedding $\mathrm{Emb}_{\mu}: G\to \mathcal{C}^{\mu}$ and the homomorphism $\mathrm{Emb}_{\mathrm{ad}(\mu)}: U_-\to \mathcal{K}_1$ form a morphism of group action data, $(G, U_-)\to (\mathcal{C}^\mu_1, \mathcal{K}_1)$, we obtain a morphism of $k$-schemes, $\tilde{\alpha}: G/U_-\to \mathcal{C}^{\mu}_1$, which clearly factors through $G\mu(t)\mathcal{K}_1/\mathcal{K}_1\subseteq \mathcal{C}_1^{\mu}$. The induced morphism $\tilde{\alpha}: G/U_-\to G\mu(t)\mathcal{K}_1/\mathcal{K}_1 $ is injective by (2) and hence is a bijection as the surjectivity is obvious.
	\end{proof}
	
	\subsection{The Varieties \( G\,\mu(t)\,\mathcal{K}_1/\mathcal{K}_1 \subseteq \prescript{}{1}{\mathcal{C}_1^\mu} \) are Strongly Quasi-Affine}
\label{S:Grosshans}

As the preceding discussion identifies this embedding with \( G/U_- \hookrightarrow \mathsf{E}_{\mu} \backslash G^2 \), we now state the relevant properties of these varieties (our presentation follows, in part, the work of J. Wang \cite{WangRedMod}). In this subsection, we assume that our base field \( k \) is algebraically closed.

We first recall some definitions. A \( k \)-variety \( X \) is said to be \emph{strongly quasi-affine} if the ring 
\(
\Gamma(X, \mathcal{O}_{X})
\)
of global sections of \( X \) is a finitely generated \( k \)-algebra and the canonical morphism
\[
X \to X_{\mathrm{aff}} := \mathrm{Spec}\,\Gamma(X, \mathcal{O}_X)
\]
is an open immersion (the latter condition is equivalent to saying that \( X \) is quasi-affine). A (closed) subgroup \( H \subseteq G \) with the property that \( G/H \) is strongly quasi-affine is also called a \emph{Grosshans subgroup} (see \cite[Pg.~22]{GrosshansBook}). Let \( X \) be a normal \( k \)-variety with (right) \( G \)-action; then \( X \) is said to be \( G \)-\emph{spherical} if \( X \) has a dense open \( B \)-orbit for some Borel subgroup \( B \) of \( G \).

\begin{proposition}
\label{Prop:qaffine}The following holds:
\begin{enumerate}
  \item The \( k \)-varieties \( G/U_- \) and \( \prescript{}{1}{\mathcal{C}_1^\mu} \) are strongly quasi-affine; i.e., \( U_- \) is a Grosshans subgroup of \( G \), and \( \mathsf{E}_{\mu} \) is a Grosshans subgroup of \( G^2 \).
  
  \item The morphism 
  \(
  \big(G/U_-\big)_{\mathrm{aff}} \hookrightarrow \big({\prescript{}{1}{\mathcal{C}_1^\mu}}\big)_{\mathrm{aff}}
  \)
  induced by the closed immersion \( G/U_- \hookrightarrow \prescript{}{1}{\mathcal{C}_1^\mu} \) is again a closed immersion.
  
  \item The varieties \( G/U_- \) and \( \big(G/U_-\big)_{\mathrm{aff}} \) are \( G\times_k M \)-spherical, and the varieties \( \prescript{}{1}{\mathcal{C}_1^\mu} \) and \( \big(\prescript{}{1}{\mathcal{C}_1^\mu}\big)_{\mathrm{aff}} \) are \( G^2 \)-spherical.
\end{enumerate}
\end{proposition}

\begin{proof}
\begin{enumerate}
  \item The fact that \( G/U_- \) is strongly quasi-affine is well-known (see \cite[Theorem 16.4]{GrosshansBook}; for the quasi-affineness, one may also see \cite[Proposition 14.26]{MilneAlgebraicGroups}). Here the fact that \( \Gamma(G/U_-, \mathcal{O}_{G/U_-}) \) is a finitely generated \( k \)-algebra relies on a theorem of Grosshans \cite{Grosshans83} which states that in any characteristic, if \( U \subseteq H \) is the unipotent radical of a parabolic subgroup of a reductive \( k \)-group \( H \), then
  \[
  \Gamma(H/U, \mathcal{O}_{H/U}) = \Gamma(H, \mathcal{O}_H)^{U}
  \]
  is a finitely generated \( k \)-algebra. Applying the same theorem to the unipotent radical \( U_-\times_k U_+ \) of the parabolic subgroup \( P_-\times_k P_+ \subseteq G^2 \), we see that
  \(
  \Gamma(G^2, \mathcal{O}_{G^2})^{U_-\times_k U_+}
  \)
  is a finitely generated \( k \)-algebra, and hence so is
  \[
  \Gamma(\mathsf{E}_{\mu}\backslash G^2, \mathcal{O}_{\mathsf{E}_{\mu}\backslash G^2})
  = \Gamma(G^2, \mathcal{O}_{G^2})^{\mathsf{E}_{\mu}}
  = \Big(\Gamma(G^2, \mathcal{O}_{G^2})^{U_-\times_k U_+}\Big)^{\Delta(M)},
  \]
  since \( \Delta(M) \), denoting the image of \( M \) under the diagonal embedding \( \Delta: G \to G^2 \), is reductive.

  The quasi-affineness of \( \mathsf{E}_{\mu}\backslash G^2 \) follows from \cite[Proposition 2.4.4]{DrGConstant} (the fiber over \( 0 \in \mathbb{A}^1 \) of the space \( (\mathbb{A}^1\times G\times G)/\tilde{G} \) there is isomorphic to \( \mathsf{E}_{\mu}\backslash G^2 \)). Below we provide another proof. Recall that \( \mathsf{E}_{\mu}\backslash G^2 \) is in fact the geometric quotient of \( G^2 \) by \( \mathsf{E}_{\mu} \), and is a separated smooth \( k \)-variety. Secondly, note that we have the following equality
  \[
  \mathsf{E}_{\mu}\backslash G^2 = (U_-\times_k U_+)\backslash\Big(\Delta(M)\backslash G^2\Big),
  \]
  where the quotient \( \Delta(M)\backslash G^2 \) is a smooth affine \( k \)-variety since \( \Delta(M) \) is a reductive subgroup of \( G^2 \); said in another way, \( \mathsf{E}_{\mu}\backslash G^2 \) is also the geometric quotient of the smooth affine variety \( \Delta(M)\backslash G^2 \) by the connected unipotent group \( U_-\times_k U_+ \). Hence if we write \( X := \Delta(M)\backslash G^2 \), by \cite[Corollary 1.5]{Fauntleroy85} we have
  \[
  X = X^s(U_-\times_kU_+),
  \]
  where \( X^s(U_-\times_kU_+) \) denotes the subset of stable points of \( X \) with respect to the action of the connected unipotent group \( U_-\times_kU_+ \) (see loc. cit. for the definition). Finally, by the proof of Theorem 1.4 of loc. cit. (taking the open set \( V \) there to be the whole \( X \)) one concludes that \( \mathsf{E}_{\mu}\backslash G^2 \) is indeed quasi-affine.

  \item This is part of \cite[Theorem 4.1.4]{WangRedMod}. As remarked in loc. cit. (the paragraph above the theorem), in case \( k \) is of characteristic \( 0 \), this can also be deduced from \cite[Proposition 5]{ADCanEmb}.

  \item We only show here the case of
  \(
  \prescript{}{1}{\mathcal{C}_1^\mu} = \mathsf{E}_\mu\backslash G^2,
  \)
  as the case of \( G/U_- \) can be argued similarly. Choose a Borel subgroup \( B_+ \) of \( G \) that is contained in \( P_+ \) and denote by \( B_- \) its opposite parabolic subgroup. Then we have \( B_- \subseteq P_- \) and \( U_\pm \subseteq \mathrm{R}_uB_{\pm} \). We claim that the \( B_+\times_k B_- \)-orbit of \( (1, 1) \in \mathsf{E}_\mu\backslash G^2 \) is open and dense.

  The case where \( P_+ = P_- = M = G \) is clear, as this follows from the fact that the image of the multiplication map \( B_+\times B_- \to G \) is open and dense in \( G \). For the general case, we need to show that the image of \( B_+\times B_- \subseteq G\times G \) in \( \mathsf{E}_{\mu}\backslash G^2 \) under the canonical projection \( G^2\to \mathsf{E}_{\mu}\backslash G^2 \) is dense open. As \( P_-U_+ = U_- M U_+ \subseteq G \) is dense open, it is enough to show that the image of \( U_+(B_+\cap M) \times U_-(B_-\cap M) \subseteq B_+\times B_- \) in the dense open set
  \(
  \mathsf{E}_{\mu}\backslash \big(U_-P_+\times U_+P_-\big)
  \)
  of \( \mathsf{E}_{\mu}\backslash G^2 \) is dense open. Clearly, we have the equality
  \[
  \mathsf{E}_{\mu}\backslash \big(U_-P_+\times U_+P_-\big)
  = \Delta(M)\backslash \big(U_+M \times U_-M\big).
  \]
  Hence we are reduced to the case where \( P_+ = P_- = M \) since \( (B_+\cap M, B_-\cap M) \) form a pair of opposite Borel subgroups of the reductive group \( M \).

  We now check that
  \(
  \big(\mathsf{E}_{\mu}\backslash G^2\big)_{\mathrm{aff}}
  \)
  is \( G^2 \)-spherical: clearly the \( \mathsf{E}_{\mu} \)-action on \( \mathsf{E}_{\mu}\backslash G^2 \) extends uniquely to an action on
  \(
  \big(\mathsf{E}_{\mu}\backslash G^2\big)_{\mathrm{aff}}.
  \)
  The normality of \( \big(\mathsf{E}_{\mu}\backslash G^2\big)_{\mathrm{aff}} \) follows from that of \( \mathsf{E}_{\mu}\backslash G^2 \) (see \cite[tag 0358]{stacks-project}). Since \( \mathsf{E}_{\mu}\backslash G^2 \) is dense open in \( \big(\mathsf{E}_{\mu}\backslash G^2\big)_{\mathrm{aff}} \), every dense open orbit of \( B_+\times_k B_- \) in \( \mathsf{E}_{\mu}\backslash G^2 \) is also dense open in \( \big(\mathsf{E}_{\mu}\backslash G^2\big)_{\mathrm{aff}} \).
\end{enumerate}
\end{proof}

\section{The Moduli Stack {$[\prescript{}{1}{\mathcal{C}_1^\mu}/\mathrm{Ad}_\tau G]$} as a Zip Stack}\label{S:ShtukAsZip}

In this section, we show that via the isomorphism $\psi: \mathsf{E}_{\mu}\backslash G^2 \cong \prescript{}{1}{\mathcal{C}_1^\mu}$, quotient stacks of the form $[\prescript{}{1}{\mathcal{C}_1^\mu}/\mathrm{Ad}_{\tau}G]$ are zip stacks—that is, they are isomorphic to quotient stacks of $G$ by zip groups determined by $\mu$ (see \cite[Definition 3.1]{PinkWedhornZiegler1}). The theory of zip stacks was developed in \cite{Moonen&WedhornDiscreteinvariants, PinkWedhornZiegler1, PinkWedhornZiegler2}.

\subsection{The Stack $[\prescript{}{1}{\mathcal{C}_1^\mu}/\mathrm{Ad} G]$ as a Zip Stack}\label{S:zipstacknotwist}

	In this subsection, let $k$ be an arbitrary field. While our primary focus in this paper is on the stack $[\prescript{}{1}{\mathcal{C}_1^\mu}/\mathrm{Ad}{\tau}G]$ for the case where $k$ has characteristic $p$ and $\tau$ is a nontrivial Frobenius of $G$, we carefully discuss the interpretation of $[\prescript{}{1}{\mathcal{C}_1^\mu}/\mathrm{Ad}_{\tau}G]$ as a zip stack when $\tau=\mathrm{id}_{\mathcal{L}^+ G}$. For simplicity, we omit $\tau$ in the notation and write $[\prescript{}{1}{\mathcal{C}_1^\mu}/\mathrm{Ad} G]$ instead. The stack $[\prescript{}{1}{\mathcal{C}_1^\mu}/\mathrm{Ad}G]$ has the advantage of being meaningful for a field $k$ of arbitrary characteristic, revealing many symmetries.

The \emph{zip stack} $[G/\mathsf{E}_{\mu}]$ is defined as the quotient stack of $G$, obtained by the right action of $\mathsf{E}_{\mu}$ on $G$, given by
\begin{equation}\label{Eq:EmuActG}
	G \times_k \mathsf{E}_{\mu} \to G, \quad \big(g, (p_-, p_+)\big) \mapsto p_+^{-1} g p_-.
\end{equation}
The natural inclusion $U_- \subseteq \mathsf{E}_{\mu}$ realizes $U_-$ as a normal subgroup of $\mathsf{E}_{\mu}$, whose induced action on $G$ is given by right multiplication. Passing to the quotient gives us an induced action of $P_+$ on $G / U_-$, which on local sections is given by
\[ G / U_- \times_k P_+ \to G / U_-, \quad (g, p_+) \mapsto p_+^{-1} g m,\]
where $m$ is the Levi component of $p_+$. This action is well-defined since $M \subseteq P_-$ normalizes $U_-$. Dually, by taking the quotient by $U_+$, we obtain a right action of $P_-$ on $U_+ \backslash G$,
\[(U_+ \backslash G) \times_k P_- \to U_+ \backslash G, \quad (g, p_-) \mapsto m^{-1}g p_-,\]
where $m$ is the Levi component of $p_-$. 
We denote $[(G/U_-)/P_+]$ and $[(U_+ \backslash G)/P_-]$ as the corresponding quotient stacks. Since the actions of $U_-$ and $U_+$ on $G$ are free, it follows from the basic theory of algebraic stacks that the projections $G \to G / U_-$ and $G \to U_+ \backslash G$ induce canonical isomorphisms of algebraic $k$-stacks,
\[
[(U_+ \backslash G)/P_-] \xleftarrow{\cancong} [G/\mathsf{E}_{\mu}] \xrightarrow{\cancong} [(G/U_-)/P_+].
\]
As preparation for the next lemma, we also consider the action of $\mathsf{E}_{\mu} \times_k G$ on $G^2$ as follows,
\begin{equation}\label{TwoGrpAct}
	G^2 \times_k (\mathsf{E}_{\mu} \times_k G) \to G^2, \quad \big((h_1, h_2), (p_-, p_+, g)\big) \mapsto \big(p_-^{-1} h_1 g, p_+^{-1} h_2 g\big).
\end{equation}
Note that both $G \subseteq \mathsf{E}_{\mu} \times_k G$ and $\mathsf{E}_{\mu} \subseteq \mathsf{E}_{\mu} \times_k G$ are normal subgroups, and their induced commuting actions on $G^2$ are free. Note also that the induced action of $\mathsf{E}_{\mu}$ on $G^2$ coincides with the one considered in \S \ref{S:RepCIImu}, and the induced action of $G$ on $G^2$ is given by right multiplication diagonally. Hence the natural projections $G^2 \to G^2 / \Delta G$ and $G^2 \to \mathsf{E}_{\mu} \backslash G^2$ induce canonical isomorphisms of algebraic $k$-stacks,
\[
[(G^2 / \Delta G)/\mathsf{E}_{\mu}] \xleftarrow{\cancong} [G^2 / (\mathsf{E}_{\mu} \times_k G)] \xrightarrow{\cancong} [(\mathsf{E}_{\mu} \backslash G^2) / \Delta G].
\]

Here we use $G^2/\Delta G$ to denote the quotient scheme of $G^2$ obtained from the diagonal right multiplication of  $G$.

Write 
\(
\upsilon: G^2 / \Delta G \cong G,   
\)
for the isomorphism of $k$-schemes, defined by $(g, h) \mapsto h g^{-1}$. Then it is easy to verify that the induced action of $\mathsf{E}_{\mu}$ on $G^2 / \Delta G$ corresponds to the action induced from \eqref{Eq:EmuActG} via the isomorphism $\upsilon$, which yields the following isomorphism,
\[
[\upsilon]: [(G^2 / \Delta G)/\mathsf{E}_{\mu}] \cong [G / \mathsf{E}_{\mu}]. 
\]

On the other hand, one verifies that the induced action of $G$ on $\mathsf{E}_{\mu} \backslash G^2$ corresponds, under the isomorphism $\psi: \mathsf{E}_{\mu} \backslash G^2 \cong \prescript{}{1}{\mathcal{C}_1^\mu}$, precisely to the conjugation action of $G$ on $\prescript{}{1}{\mathcal{C}_1^\mu}$. Thus, we obtain the isomorphism,
\[
[\psi]: [(\mathsf{E}_{\mu} \backslash G^2) / \Delta G] \cong [\prescript{}{1}{\mathcal{C}_1^\mu} / \mathrm{Ad} G].
\]

The composition of these isomorphisms gives us the isomorphism \[[\prescript{}{1}{\mathcal{C}_1^\mu}/\mathrm{Ad} G] \cong [G / \mathsf{E}_{\mu}],\]
which identifies the moduli stack $[\prescript{}{1}{\mathcal{C}_1^\mu}/\mathrm{Ad} G]$ as the zip stack $[G / \mathsf{E}_{\mu}]$.

\begin{proposition}\label{Prop:ManyIsom}
	The following holds true:
	\begin{enumerate}
		\item The closed embedding $\alpha: G/U_- \to \prescript{}{1}{\mathcal{C}_1^\mu}$ (Corollary \ref{Cor:Embed}) is compatible with the embedding $P_+ \hookrightarrow G$ with respect to the corresponding group actions, thus inducing an isomorphism of algebraic $k$-stacks,
		\[
		[\alpha]: [(G/U_{-})/P_{+}] \xlongrightarrow{\cong} [\prescript{}{1}{\mathcal{C}_1^\mu}/\mathrm{Ad} G].
		\]
		
		\item The closed embedding $\beta: U_+ \backslash G \to \prescript{}{1}{\mathcal{C}_1^\mu}$ (Corollary \ref{Cor:Embed}) is compatible with the embedding $P_- \hookrightarrow G$ and the corresponding group actions, thus inducing an isomorphism of algebraic $k$-stacks,
		\[
		[\beta]: [(U_+ \backslash G)/P_{-}] \xlongrightarrow{\cong} [\prescript{}{1}{\mathcal{C}_1^\mu}/\mathrm{Ad} G].
		\]
		
		\item The following diagram of isomorphisms between algebraic $k$-stacks is 2-commutative:
		\[
		\xymatrix{
			[G/\mathsf{E}_{\mu}]\ar[r]^{\cancong}& [(U_+\backslash G)/P_-]\ar[r]^{[\beta]}_{\cong}&[\prescript{}{1}{\mathcal{C}_1^\mu}/\mathrm{Ad} G] \\
			[(G^2/\Delta G)/\mathsf{E}_{\mu}]\ar[d]_{[\upsilon]}^{\cong}\ar[u]^{[\upsilon]}_{\cong}&[G^2/(\mathsf{E}_{\mu}\times_k G)]\ar[l]_{\cancong}\ar[r]^{\cancong}&[(\mathsf{E}_{\mu}\backslash G^2)/\Delta G]\ar[d]^{[\psi]}_{\cong}\ar[u]_{[\psi]}^{\cong}\\
			[G/\mathsf{E}_{\mu}]\ar[r]^{\cancong}&[(G/U_-)/P_+] \ar[r]^{[\alpha]}_{\cong}&[\prescript{}{1}{\mathcal{C}_1^\mu}/\mathrm{Ad} G].
		}
		\]
	\end{enumerate}
\end{proposition}

\begin{proof}
	\begin{enumerate}
		\item Let $A$ be a $k$-algebra. Take $g \in G(A)$ and $p_+ \in P_+(A)$. Then we have the following equality:
		\[
		\begin{array}{rll}
			\alpha(g \cdot p_+) &= p_{+}^{-1} g \mu(t) m &= \Big(p_+^{-1} g (\mu(t) u_+^{-1} \mu(t)^{-1}) g^{-1} p_+\Big) \Big(p_+^{-1} g \mu(t) p_+\Big) \\
			&= p_+^{-1} g \mu(t) p_+ &= \alpha(g) \cdot p_+.
		\end{array}
		\]
		Here, the third identity holds because, according to Lemma \ref{Lem:IntegLem}, we have $\mu(t) u_+^{-1} \mu(t)^{-1} \in \mathcal{K}_1(A)$. This proves the first assertion of (1). The second assertion of (1) will follow from the commutativity of the upper square in the diagram of (3).

		\item We leave the proof to the reader, as the assertions are dual to those of (1).

		\item We first show that the upper square of the diagram is $2$-commutative. We consider the embedding $i_2: G \to G^2$ above and define an embedding of algebraic $k$-groups as follows:
		\[
		j_1: \mathsf{E}_{\mu} \hookrightarrow \mathsf{E}_{\mu} \times_k G, \quad (p_-, p_+) \mapsto (p_-, p_+, p_-).
		\]
		One can verify that the pair $(i_2, j_1)$ is compatible with the action of $\mathsf{E}_{\mu}$ on $G$ and the action of $\mathsf{E}_{\mu} \times_k G$ on $G^2$ defined in \eqref{TwoGrpAct}, thus forming a morphism of group action data from $(G, \mathsf{E}_{\mu})$ to $(G^2, \mathsf{E}_{\mu} \times_k G)$. It can be directly checked that the induced morphism
		$[i_2]: [\mathsf{E}_{\mu} \backslash G] \to [G^2 / (\mathsf{E}_{\mu} \times_k G)]$ provides \emph{an inverse} of the composition,
		\begin{equation}\label{Composition}
			[G^2 / (\mathsf{E}_{\mu} \times_k G)] \xrightarrow{\cancong} [\big(G^2 / \Delta G\big) / \mathsf{E}_{\mu}] \xrightarrow[\cong]{[\upsilon]} [G / \mathsf{E}_{\mu}].
		\end{equation} 
		Hence we need to show that the two isomorphisms from $[G / \mathsf{E}_{\mu}] \to [\prescript{}{1}{\mathcal{C}_1^\mu}/\mathrm{Ad} G]$,
		$[\psi] \circ \mathsf{can} \circ [i_2]$ and $[\beta] \circ \mathsf{can}$ are 2-commutative. It is straightforward to verify that both isomorphisms are induced by the morphism of group action data:
		\[
		(b, q_1): (G, \mathsf{E}_{\mu}) \to (\prescript{}{1}{\mathcal{C}_1^\mu}, G),
		\]
		where $b: G \to \prescript{}{1}{\mathcal{C}_1^\mu}$ is given by $g \mapsto \mu(t) g$, and $q_1: \mathsf{E}_{\mu} \to G$ is given by $(p_-, p_+) \mapsto p_-$.
        
        The 2-commutativity of the lower square follows from a dual argument, noting that
		\[
		i_1: G \to G^2, \ g \mapsto (g^{-1}, 1), \quad j_2: \mathsf{E}_{\mu} \hookrightarrow \mathsf{E}_{\mu} \times G, \ (p_-, p_+) \mapsto (p_-, p_+, p_+),
		\]
		form a morphism of group action data $(i_1, j_2): (G, \mathsf{E}_{\mu}) \to (\prescript{}{1}{\mathcal{C}_1^\mu}, G)$, whose induced morphism of $k$-stacks $[i_1]: [G/\mathsf{E}_{\mu}] \to [\prescript{}{1}{\mathcal{C}_1^\mu}/\mathrm{Ad} G]$ provides \emph{another inverse} of the composition \eqref{Composition}. The remaining arguments are left to the reader.
	\end{enumerate}
\end{proof}

\subsection{The Stack $[\prescript{}{1}{\mathcal{C}_1^\mu}/\mathrm{Ad}_{\tau}G]$ as a Zip Stack}\label{S:zipstack}

In this subsection, let $k$ be a field of characteristic $p$, and assume further that $G$ is defined over $\mathbb{F}_p \subseteq k$. Due to this assumption, we have canonical isomorphisms
\[
G^{\sigma} \cancong G, \quad \mathcal{L}^+ G \cancong \mathcal{L}^+ G^{\sigma}, \quad \mathcal{L} G \cancong \mathcal{L} G^{\sigma}.
\]
Consider 
\(
\tau = \sigma^{n}: G \to G^{\sigma^n} \cancong G,
\)
the $n$-th iteration of $\sigma: G \to G$, for some $n \geq 0$. Let $\tau: \mathcal{L}^+ G \to \mathcal{L}^+ G$ and $\tau: \mathcal{L} G \to \mathcal{L} G$ be the induced Frobenii as in \eqref{Sigofloopgp}. 

Denote by $[G/\mathsf{R}_{\tau}\mathsf{E}_{\mu}]$ the quotient stack of $G$, obtained from the following \emph{right partial Frobenius multiplication} of $\mathsf{E}_{\mu}$ on $G$:
\begin{equation}\label{Eq:EmuActGTau}
	\mathsf{R}_\tau: G \times_k \mathsf{E}_{\mu} \to G, \quad \big(g, (p_-, p_+)\big) \mapsto p_+^{-1} g \tau(p_-).
\end{equation}
The natural inclusion $U_+ \subseteq \mathsf{E}_{\mu}$ realizes $U_+$ as a normal subgroup of $\mathsf{E}_{\mu}$, whose induced right action on $G$ is given by left multiplication. Passing to the quotient by $U_+$, the group action \eqref{Eq:EmuActGTau} becomes,
\begin{equation}\label{Eq:PminusAction}
   (U_+ \backslash G) \times_k P_- \to U_+ \backslash G, \quad (g, p_-) \mapsto m^{-1} g \tau(p_-), 
\end{equation}

where $m$ is the Levi component of $p_-$.
This action is well-defined since $M \subseteq P_+$ normalizes $U_+$.

We denote $[(U_+\backslash G)/_{\tau}P_-]$ as the corresponding quotient stacks. Since the action of $U_+$ on $G$ is free, the canonical projection $G \to U_+ \backslash G$ induces an isomorphism of algebraic $k$-stacks:
\[
[G/\mathsf{R}_{\tau}\mathsf{E}_{\mu}] \xrightarrow{\cancong} [(U_+\backslash G)/_{\tau}P_-].
\]
As in the previous subsection, we consider a right action of $\mathsf{E}_{\mu}\times_k G$ on $G^2$ as follows:
\begin{equation}\label{GpActDelta}
	\delta: G^2 \times_k (\mathsf{E}_{\mu} \times_k G) \to G^2, \quad \big((h_1, h_2), (p_-, p_+, g)\big) \mapsto \big(p_-^{-1} h_1 g, p_+^{-1} h_2 \tau(g)\big).
\end{equation}
Let us denote $[G^2/(\mathsf{E}_{\mu} \times_k G)]$ as the resulting quotient stack over $k$. 

Note that both 
\(
G \subseteq \mathsf{E}_{\mu} \times_k G\) and \(\mathsf{E}_{\mu} \subseteq \mathsf{E}_{\mu} \times_k G
\)
are normal subgroups whose commuting actions on \(G^2\) are free. Note also that the \(\mathsf{E}_{\mu}\)-action on \(G^2\) induced by \(\delta\) coincides with that considered in \S \ref{S:RepCIImu},  Thus, the natural projections 
\(
G^2 \to G^2/_{\delta}G\) and \( G^2 \to \mathsf{E}_{\mu} \backslash G^2
\)
induce canonical isomorphisms of algebraic \(k\)-stacks:
\[
\Bigl[\bigl(G^2/_{\delta}G\bigr)/\mathsf{E}_{\mu}\Bigr]
\xleftarrow{\cancong} \bigl[G^2/(\mathsf{E}_{\mu}\times_k G)\bigr] 
\xrightarrow{\cancong} \Bigl[(\mathsf{E}_{\mu}\backslash G^2)/_{\delta}G\Bigr].
\]
Here we use  \(G^2/_{\delta}G\) to denote the quotient scheme of \(G^2\) under the induced \(G\)-action. 

Write 
\(
\rho: G^2/_{\delta}G \cong G
\)
for the isomorphism of $k$-schemes given by $ (g,h) \mapsto h\,\tau(g)^{-1}$. It is straightforward to verify that the induced \(\mathsf{E}_{\mu}\)-action on \(G^2/_{\delta}G\) coincides via \(\rho\) with the one defined in \eqref{Eq:EmuActGTau}. Hence, we obtain the isomorphism of \(k\)-stacks
\[
[\rho]: \Bigl[(G^2/_{\delta}G)/\mathsf{E}_{\mu}\Bigr] \cong \bigl[G/\mathsf{R}_{\tau}\mathsf{E}_{\mu}\bigr].
\]

Similarly, one verifies that the \(G\)-action on \(\mathsf{E}_{\mu}\backslash G^2\) induced by \(\delta\) corresponds, via the isomorphism 
\(
\psi: \mathsf{E}_{\mu}\backslash G^2 \cong \prescript{}{1}{\mathcal{C}_1^\mu},
\)
to the \(\tau\)-conjugation action of \(G\) on \(\prescript{}{1}{\mathcal{C}_1^\mu}\). Therefore, we also have the isomorphism
\[
[\psi]: \Bigl[(\mathsf{E}_{\mu}\backslash G^2)/_{\delta}G\Bigr] \cong \Bigl[\prescript{}{1}{\mathcal{C}_1^\mu}/\mathrm{Ad}_{\tau}G\Bigr].
\]

\begin{lemma}\label{Lem:ManyIsom2}The following holds:
	\begin{enumerate}
		\item The closed embedding $\beta: U_+ \backslash G \hookrightarrow \prescript{}{1}{\mathcal{C}_1^\mu}$ (Corollary \ref{Cor:Embed}) and the inclusion $P_- \hookrightarrow G$ are compatible with the action of $P_-$ on $U_+ \backslash G$ in \eqref{Eq:PminusAction} and the $\tau$-conjugation action of $G$ on $\prescript{}{1}{\mathcal{C}_1^\mu}$, thus inducing an isomorphism of algebraic $k$-stacks:
		\[
		[\beta]: [(U_+ \backslash G)/_{\tau}P_{-}] \xlongrightarrow{\cong} [\prescript{}{1}{\mathcal{C}_1^\mu}/\mathrm{Ad}_{\tau} G].
		\] 
		
		\item The following diagram of isomorphisms between algebraic $k$-stacks is 2-commutative:
		\begin{equation*}
			\xymatrix{[ (G^2/_{\delta} G)/\mathsf{E}_{\mu}]\ar[d]_{[\rho]}^{\cong}&[G^2/(\mathsf{E}_{\mu}\times_k G)]\ar[l]_{\cancong}\ar[r]^{\cancong}&[(\mathsf{E}_{\mu}\backslash G^2)/_{\delta}G]\ar[d]^{[\psi]}_{\cong}\\
				[G/\mathsf{R}_{\tau}\mathsf{E}_{\mu}]\ar[r]^{\cancong}& [(U_+\backslash G)/_{\tau}P_-]\ar[r]^{[\beta]}_{\cong}&[\prescript{}{1}{\mathcal{C}_1^\mu}/\mathrm{Ad}_{\tau} G].}
		\end{equation*}
		In particular, we have an isomorphism of zero-dimensional smooth Artin $k$-stacks:
		\[
		[G/\mathsf{R}_{\tau}\mathsf{E}_{\mu}] \cong [\prescript{}{1}{\mathcal{C}_1^\mu}/\mathrm{Ad}_{\tau} G].
		\]
	\end{enumerate}
\end{lemma}

\begin{proof}
	\begin{enumerate}
		\item This is an easy calculation similar to Lemma \ref{Prop:ManyIsom}.(1). Let $A$ be a $k$-algebra, $g \in G(A)$, and $p_- \in P_-(A)$. We have the following equality:
		\[
		\begin{array}{rll}
			\beta(g \cdot p_-) &= \mu(t) m^{-1} g \tau(p_-) &= \Big(p_-^{-1} \mu(t) g \tau(p_-)\Big) \Big(\tau(p_-)^{-1} g^{-1} \mu(t)^{-1} u_- \mu(t) g \tau(p_-)\Big) \\
			&= p_-^{-1} \mu(t) g \tau(p_-) &= \beta(g) \cdot p_-.
		\end{array}
		\]
		
		\item The proof is a word-for-word adaptation of the proof of \ref{Prop:ManyIsom}.(3). 
	\end{enumerate}
\end{proof}

\begin{remark}
In comparison with Lemma~\ref{Prop:ManyIsom}, one may expect a loss of symmetry when the \(\tau\)-conjugation is nontrivial. Note that when \(n=0\), Lemma~\ref{Lem:ManyIsom2} coincides with Lemma~\ref{Prop:ManyIsom}.
\end{remark}

\begin{remark}
Let \(\tau=\sigma\). Advanced readers may notice the similarity between \([G/\mathsf{R}_{\sigma}\mathsf{E}_{\mu}]\) and the \(G\)-zip stack \({G\text{-}\mathsf{Zip}}^{\mu}\) defined in \cite{PinkWedhornZiegler2}, as the latter is isomorphic to the quotient stack \([G/\mathsf{E}_{\mu}^{\mathrm{PWZ}}]\), where \(\mathsf{E}_{\mu}^{\mathrm{PWZ}}\) is a Frobenius-twisted variant of \(\mathsf{E}_{\mu}\). Their relationship will be explored in Section~\ref{S:Gzips}.
\end{remark}

\subsection{Moduli Stacks of 1-1-Truncated Local $G$-Shtukas vs $G$-Zip Stacks}\label{S:Gzips}
In this section, we let $G$ and $k$ be as in \S \ref{S:zipstack}, and take $\tau=\sigma: \mathcal{L}^+G \to \mathcal{L}^+G$. Clearly the stack $[\prescript{}{1}{\mathcal{C}_1^\mu}/\mathrm{Ad}_{\sigma}G]$ receives a canonical projection from the stack 
\[
[\mathcal{C}^{\mu}/\mathrm{Ad}_{\sigma}\mathcal{L}^+ G],
\] 
which is the \emph{classifying stack of local $G$-shtukas of type $\mu$}. We call $[\prescript{}{1}{\mathcal{C}_1^\mu}/\mathrm{Ad}_\sigma G]$ the \emph{(coarse) moduli stack of 1-1 truncated local $G$-shtukas} of type $\mu$; cf. Remark \ref{Rmk:trunSch}.

The focus of this subsection is the comparison of $[\prescript{}{1}{\mathcal{C}_1^\mu}/\mathrm{Ad}_\sigma G]$ and the $G$-zip stack $[G/\mathsf{E}_{\mu}^{\mathrm{PWZ}}]$  (recalled below). Our discussion on $[G/\mathsf{E}_{\mu}^{\mathrm{PWZ}}]$ follows the same logic as in \S \ref{S:zipstacknotwist} and \S \ref{S:zipstack}. 


\begin{definition}[The Frobenius Zip Group of $\mu$]
	For a cocharacter $\mu: \mathbb{G}_{m,k} \to G$, let $\mathsf{E}_{\mu}^{\mathrm{PWZ}} \subseteq P_-^{\tau} \times_k P_+$ be the closed subgroup whose $A$-valued points for a $k$-algebra $A$ are given by 
	\begin{align*}
		\mathsf{E}_{\mu}^{\mathrm{PWZ}}(A) = \{(p_- = u_- \tau(m), p_+ = u_+ m) \;|\; m \in M(A), \; u_+ \in U_+(A), \; u_- \in U_-^{\tau}(A)\},
	\end{align*}
	where we use the decomposition $P_+ = U_+ \rtimes M$ and $P_-^{\tau} = U_-^{\tau} \rtimes M^{\tau}$.
\end{definition}

 Consider the following right action of $\mathsf{E}_{\mu}^{\mathrm{PWZ}}$ on $G$:
\begin{equation}\label{Eq:emuActGTau}
	G \times_k \mathsf{E}_{\mu}^{\mathrm{PWZ}} \to G, \quad (g, (p_-, p_+) ) \mapsto p_+^{-1} g p_-.
\end{equation}
Denote by $[G/\mathsf{E}_{\mu}^{\mathrm{PWZ}}]$ the resulting quotient algebraic $k$-stack of $G$. The natural inclusion $U_-^{\tau} \subseteq \mathsf{E}_{\mu}^{\mathrm{PWZ}}$ realizes $U_-^{\tau}$ as a normal subgroup of $\mathsf{E}_{\mu}^{\mathrm{PWZ}}$, whose induced action on $G$ is given by right multiplication. Passing to the quotient by $U_-^{\tau}$ gives us an induced action of $P_+$ on $G/U_-^{\tau}$, which on local sections is given by
\[
G/U_-^{\tau} \times_k P_+ \to G/U_-^{\tau}, \quad (g, p_+) \mapsto p_+^{-1} g \tau(m),
\]
where $m$ is the Levi component of $p_+$. This action is well-defined since $M^{\tau} \subseteq P_-^{\tau}$ normalizes $U_-^{\tau}$. Denote $[(G/U_-^{\tau})/P_+]$ as the corresponding quotient stacks. Since the action of $U_-^{\tau}$ on $G$ is free, the canonical projection $\mathrm{pr}:G \to G/U_-^{\tau}$ induces an isomorphism of algebraic $k$-stacks:
\begin{equation}\label{GpActXi}
	[G/\mathsf{E}_{\mu}^{\mathrm{PWZ}}] \xrightarrow{\cancong} [(G/U_-^{\tau})/P_+]. 
\end{equation}
Following the previous two subsections, we also consider the action of $\mathsf{E}_{\mu}^{\mathrm{PWZ}} \times_k G$ on $G^2$ as given below:
\[
G^2 \times_k (\mathsf{E}_{\mu}^{\mathrm{PWZ}} \times_k G) \to G^2, \quad \big((h_1,h_2), (p_-, p_+, g)\big) \mapsto \big(p_-^{-1} h_1 g, p_+^{-1} h_2 g\big).
\]

We denote $[G^2/(\mathsf{E}_{\mu}^{\mathrm{PWZ}} \times_k G)]$ as the resulting quotient stack over $k$. 

Note that $G \subseteq \mathsf{E}_{\mu}^{\mathrm{PWZ}} \times_k G$ and $\mathsf{E}_{\mu}^{\mathrm{PWZ}} \subseteq \mathsf{E}_{\mu}^{\mathrm{PWZ}} \times_k G$ are normal subgroups, and their induced commuting actions on $G^2$ are free. Note also that the induced action of $\mathsf{E}_{\mu}^{\mathrm{PWZ}}$ on $G^2$ coincides with the action described in \eqref{Eq:emuActGTau}, and the induced action of $G$ on the quotient scheme $\mathsf{E}_{\mu} \backslash G^2$ is by right multiplication diagonally. Hence the natural projections $G^2 \to G^2/\Delta G$ and $G^2 \to \mathsf{E}_{\mu} \backslash G^2$ induce canonical isomorphisms of algebraic $k$-stacks:
\[
[(G^2/\Delta G)/\mathsf{E}_{\mu}^{\mathrm{PWZ}}] \xleftarrow{[\upsilon]} [G^2/(\mathsf{E}_{\mu}^{\mathrm{PWZ}} \times_k G)] \xrightarrow{\cancong} [(\mathsf{E}_{\mu}^{\mathrm{PWZ}} \backslash G^2)/\Delta G].
\]

Denote by $\upsilon: G^2/\Delta G \cong G$ the isomorphism of $k$-schemes given by $(g, h) \mapsto h g^{-1}$. One verifies that the induced actions of $\mathsf{E}_{\mu}^{\mathrm{PWZ}}$ on $G^2/\Delta G$ coincides via $\upsilon$ with the action defined in \eqref{Eq:emuActGTau}. Therefore, we obtain an isomorphism of $k$-stacks, 
\[
[(G^2/\Delta G)/\mathsf{E}_{\mu}^{\mathrm{PWZ}}] \xrightarrow[\cong]{[\upsilon]}[G/\mathsf{E}_{\mu}^{\mathrm{PWZ}}].
\]

To compare $[\prescript{}{1}{\mathcal{C}_1^\mu}/\mathrm{Ad}_{\tau}G]$ with $[G/\mathsf{E}_{\mu}^{\mathrm{PWZ}}]$, we consider the following sequence of morphisms:
\[
G^2 \xrightarrow{\tau \times \mathrm{id}} G^2 \xrightarrow{\mathrm{id} \times \tau} G^2.
\]
It is straightforward to verify that this sequence is compatible with the sequence of group homomorphisms,
\[
\mathsf{E}_{\mu} \times_k G \xrightarrow{\mathrm{id} \times \tau \times \tau} \mathsf{E}_{\mu}^{\mathrm{PWZ}} \times_k G \xrightarrow{\tau \times \mathrm{id} \times \mathrm{id}} \mathsf{E}_{\tau(\mu)} \times_k G,
\]
with respect to the corresponding group actions, thus inducing a sequence of morphisms between $k$-stacks:
\begin{equation}\label{IntertwinSeq}
	[G^2/(\mathsf{E}_{\mu} \times_k G)] \xrightarrow{[\tau \times \mathrm{id}]} [G^2/(\mathsf{E}_{\mu}^{\mathrm{PWZ}} \times_k G)] \xrightarrow{[\mathrm{id} \times \tau]} [G^2/(\mathsf{E}_{\tau(\mu)} \times_k G)].
\end{equation}

\begin{proposition}\label{Prop:ManyIsom3} 
	The following holds:
	\begin{enumerate}
		\item The sequence \eqref{IntertwinSeq} induces the following 2-commutative diagram of morphisms of zero-dimensional smooth Artin $k$-stacks:
		\begin{equation*}
			\xymatrix{[\prescript{}{1}{\mathcal{C}_1^\mu}/\mathrm{Ad}_{\tau}G] \ar[rr]^{[\tau]} && [\prescript{}{1}{\mathcal{C}_1^{\tau(\mu)}}/\mathrm{Ad}_{\tau}G] \\
				[(\mathsf{E}_{\mu} \backslash G^2)/_{\delta}G] \ar[r]^{[\tau \times \mathrm{id}]}\ar[u]^{[\psi]}_{\cong} & [( \mathsf{E}^{\mathrm{PWZ}}_{\mu} \backslash G^2)/\Delta G] \ar[r]^{[\mathrm{id} \times \tau]} & [(G^2/\mathsf{E}_{\tau(\mu)})/_{\delta} G] \ar[u]^{\cong}_{[\psi]} \\
				[G^2/(\mathsf{E}_{\mu} \times_k G)] \ar[r]^{[\tau \times \mathrm{id}]}\ar[d]^{\cancong} \ar[u]_{\cancong} & [G^2/(\mathsf{E}_{\mu}^{\mathrm{PWZ}} \times_k G)] \ar[r]^{[\mathrm{id} \times \tau]}\ar[d]^{\cancong} \ar[u]_{\cancong} & [G^2/(\mathsf{E}_{\tau(\mu)} \times_k G)] \ar[d]_{\cancong} \ar[u]^{\cancong} \\
				[(G^2/_{\delta}G)/\mathsf{E}_{\mu}] \ar[r]^{[\tau \times \mathrm{id}]}\ar[d]^{\cong}_{[\rho]} & [(G^2/\Delta G)/\mathsf{E}_{\mu}^{\mathrm{PWZ}}] \ar[r]^{[\mathrm{id} \times \tau]}\ar[d]^{\cong}_{[\upsilon]} & [(G^2/_{\delta}G)/\mathsf{E}_{\tau(\mu)}] \ar[d]^{[\rho]}_{\cong} \\
				[G/\mathsf{R}_{\tau} \mathsf{E}_{\mu}] \ar[r]^{[\mathrm{id}]} & [G/\mathsf{E}_{\mu}^{\mathrm{PWZ}}] \ar[r]^{[\tau]} & [G/\mathsf{R}_{\tau} \mathsf{E}_{\tau(\mu)}].}
		\end{equation*}
		
		\item Every arrow in the diagram above becomes an isomorphism of stacks after passing to perfection.
		
		\item The closed embedding $\alpha: G/U_-^{\tau} \hookrightarrow \prescript{}{1}{\mathcal{C}_1^{\tau(\mu)}}$ (see Corollary \ref{Cor:Embed}) is compatible with the inclusion $P_+ \hookrightarrow G$, with respect to the action of $P_+$ on $G/U_-^{\tau}$ and the $G$-$\tau$ conjugation action on $\prescript{}{1}{\mathcal{C}_1^{\tau(\mu)}}$, thus inducing a morphism of algebraic $k$-stacks:
		\[
		[\alpha]: [(G/U_-^{\tau})/P_+] \longrightarrow [\prescript{}{1}{\mathcal{C}_1^{\tau(\mu)}}/\mathrm{Ad}_{\tau} G].
		\]
		
		\item The following diagram between morphisms of algebraic $k$-stacks is $2$-commutative:
		\begin{equation*}
			\xymatrix{[G/\mathsf{E}_{\mu}^{\mathrm{PWZ}}] \ar[dr]_{\gamma} \ar[rr]^{\cancong} && [ (G/U_-^{\tau})/P_+] \ar[dl]^{[\alpha]} \\
				& [\prescript{}{1}{\mathcal{C}_1^{\tau(\mu)}}/\mathrm{Ad}_{\tau} G].}
		\end{equation*}
		Here, $\gamma$ is the morphism obtained in (1) above. 

        \item Every arrow in the diagram above induces a \emph{homeomorphism} on topological spaces. 
	\end{enumerate}
\end{proposition}

\begin{proof}
	\begin{enumerate}
		\item We compute as in the proof of \ref{Prop:ManyIsom}.(1). Let $A$ be a $k$-algebra. Take $g \in G(A)$ and $p_+ \in P_+(A)$. Then we have the following equality:
		\[
		\begin{array}{rll}
			\alpha(g \cdot p_+) &= p_{+}^{-1} g \mu^{\tau}(t) \tau(m) &= \Big(p_+^{-1} g \tau(\mu(t) u_+^{-1} \mu^{\tau}(t)^{-1}) g^{-1} p_+\Big) \Big(p_+^{-1} g \mu^{\tau}(t) \tau(p_+)\Big) \\
			&= p_+^{-1} g \mu^{\tau}(t) \tau(p_+) &= \alpha(g) \cdot p_+.
		\end{array}
		\]

		\item This is a straightforward verification since the whole diagram is induced by the sequence \eqref{IntertwinSeq}. 
		\item This will follow from the commutative diagram in (4). 
        \item  Note that the composition $[\alpha] \circ \mathsf{can}$ is induced by the morphism of group action data \[(a, q_2): (G, \mathsf{E}^{\mathrm{PWZ}}_{\mu}) \to (\prescript{}{1}{\mathcal{C}_1^{\tau(\mu)}}, G),\] 
		where $a: G \to \prescript{}{1}{\mathcal{C}_1^{\tau(\mu)}}$ is given by $g\mapsto g\mu^{\tau}(t)$, and $q_2: \mathsf{E}_{\mu}^{\mathrm{PWZ}} \to G$ is given by $(p_-, p_+) \mapsto p_+$. The morphism $\gamma$ is induced by the morphism of group action data 
		\[
		(b, q_1): (G, \mathsf{E}^{\mathrm{PWZ}}_{\mu}) \to (\prescript{}{1}{\mathcal{C}_1^{\tau(\mu)}}, G),
		\] 
		where $b: G \to \prescript{}{1}{\mathcal{C}_1^{\tau(\mu)}}$ is given by $g\mapsto \mu^{\tau}(t)\tau(g)$, and where $q_1: \mathsf{E}_{\mu}^{\mathrm{PWZ}} \to G$ is given by $(p_-, p_+) \mapsto p_-$. By \S \ref{S:GpActData}, it suffices to set up a $k$-morphism $\beta: G\to G $, such that the equality \[q_2 (p_-, p_+) \beta (p_+^{-1}g p_-)= \beta(g)q_1(p_-, p_+),\] holds for all $ g\in G$ and all $(p_-, p_+)\in \mathrm{E}_{\mu}^{\mathrm{PWZ}}$. One verifies easily that taking $\beta=\mathrm{id}_G$ suffices. 
        \item It suffices to show that the sequence 
\[
[G/\mathsf{R}_{\tau}\mathsf{E}_{\mu}]\xrightarrow{[\mathrm{id}]}[G/\mathsf{E}_{\mu}^{\mathrm{PWZ}}]\xrightarrow{[\tau]}[G/\mathsf{R}_{\tau}\mathsf{E}_{\tau(\mu)}]
\]
induces homeomorphisms on the underlying topological spaces. Note that for every perfect \( k \)-algebra \( A \), the iterated Frobenii on \( A \)-valued points,
\[
\tau: G(A)\to G^{\tau}(A)=G(A),\quad \tau: P_+(A)\to P_+^{\tau}(A),\quad \tau: P_-(A)\to P_-^{\tau}(A)
\]
are bijections. It follows that the morphisms \( [\mathrm{id}] \) and \( [\tau] \) induce bijections between orbit spaces,
\[
G(A)/\mathsf{R}_{\tau}\mathsf{E}_{\mu}(A)\xrightarrow{\epsilon_1} G(A)/\mathsf{E}_{\mu}^{\mathrm{PWZ}}(A)\xrightarrow{\epsilon_2} G(A)/\mathsf{R}_{\tau}\mathsf{E}_{\tau(\mu)}(A).
\]
Note also that if \( A=k \) is algebraically closed, with all terms in the preceding sequence equipped with the quotient topology of \( G(k) \), the maps \( \epsilon_1 \) and \( \epsilon_2 \) are identified with \( [\mathrm{id}]^{\mathrm{top}} \) and \( [\tau]^{\mathrm{top}} \) respectively. Using the fact that the iterated Frobenius \( G(k)\xrightarrow{\tau} G(k) \) is a homeomorphism and that each orbit space is equipped with the quotient topology of \( G(A) \), one easily deduces that both \( \epsilon_1 \) and \( \epsilon_2 \) are homeomorphisms.

	\end{enumerate}
\end{proof}

\begin{corollary}[Comparison Between Shtukas and Zips]\label{Cor:Comp0}
	Suppose that $k$ is a finite field. Then we have a chain of algebraic $k$-stacks:
	\[
	\cdots \to [\prescript{}{1}{\mathcal{C}_1^{\tau^i(\mu)}}/\mathrm{Ad}_{\tau}G] \to [G/\mathsf{E}^{\mathrm{PWZ}}_{\tau^{i}(\mu)}] \to [\prescript{}{1}{\mathcal{C}_1^{\tau^{i+1}(\mu)}}/\mathrm{Ad}_{\tau}G] \to \cdots, \quad i \in \mathbb{Z}.
	\]

	Moreover, each connecting morphism in the above sequence induces a \emph{homeomorphism} on the underlying topological spaces. 
    
    In particular, if $k=\mathbb{F}_{p^n}$ and $\tau = \sigma$, we have a cyclic chain of homeomorphisms of algebraic stacks over $k$, whose perfections become isomorphisms:
	\[
	[\prescript{}{1}{\mathcal{C}_1^{\mu}}/\mathrm{Ad}_{\sigma}G] \to G\text{-}\mathsf{Zip}^{\mu} \to [\prescript{}{1}{\mathcal{C}_1^{\sigma(\mu)}}/\mathrm{Ad}_{\sigma}G] \to \cdots \to [\prescript{}{1}{\mathcal{C}_1^{\sigma^n(\mu)}}/\mathrm{Ad}_{\sigma}G] = [\prescript{}{1}{\mathcal{C}_1^{\mu}}/\mathrm{Ad}_{\sigma}G].
	\]
\end{corollary}

\begin{proof}
	This is a formal consequence of (1) and (2) in Proposition \ref{Prop:ManyIsom3} by repeatedly replacing $\mu$ with its Frobenius twist.
\end{proof}

\begin{remark}
  Intuitively, the zip stack ${G\text{-}\mathsf{Zip}}^{\mu}$ classifies 1-truncated $p$-divisible groups of type~$\mu$. The preceding corollary justifies $[\prescript{}{1}{\mathcal{C}_1^\mu}/\mathrm{Ad}_{\sigma}G]$ as the function-field analog of ${G\text{-}\mathsf{Zip}}^{\mu}$ and also indicates that ${G\text{-}\mathsf{Zip}}^{\mu}$ can be viewed as the partial Frobenius twist of its function-field counterpart. In particular, they share the same underlying topological spaces.
\end{remark}

\begin{corollary}\label{Cor:Comp}
	Let us consider the situation described in Remark \ref{Rmk:EqvsMix}. We then have the following summarizing commutative diagram of finite topological spaces:
	\begin{equation}\label{Eq:diamond}
		\xymatrix{ & {}^JW \ar[dl]_{\epsilon_1} \ar[dr]^{\epsilon_3} \ar[d]^{\epsilon_2} & \\ 
			[\prescript{}{1}{\mathcal{C}_1^{\sigma(\mu)}}/\mathrm{Ad}_{\sigma}G](k) \ar[dr]_{\cancong} & {G\text{-}\mathsf{Zip}}^{\mu}(k) \ar[l]_{\quad \quad \omega} \ar[r]^{\omega'\quad\quad} & [\prescript{}{1}{\mathrm{C}_1^{\sigma(\mu)}}/\mathrm{Ad}_{\sigma}G](k) \ar[dl]^{\cancong}. \\
			& [G/\mathsf{R}_{\sigma} \mathsf{E}_{\mu}](k) \ar[u]^{\cancong} & }   
	\end{equation}
\end{corollary}

\begin{proof}
	The commutativity of the left part of the diagram follows by construction. By Remark \ref{Rmk:EqvsMix}, the right side of diagram \eqref{Eq:diamond} is identified with the left part.
\end{proof}
\begin{remark}  We note the following:
\begin{enumerate}
    \item Without Theorem \ref{Thm:eqvsmix}, identifying $\omega$ and $\omega'$ independently of $\omega_1,\omega_3$ was unclear. These bijections were constructed in \cite{Wortmann} ( for $\omega'$) and \cite{Yan18} ( for $\omega$), but neither addressed topology.
    \item While $\epsilon_1,\epsilon_2$ are nontrivial homeomorphisms, Corollary \ref{Cor:Comp} shows their homeomorphism assertions are equivalent.
\end{enumerate}
\end{remark}

\subsection{The Moduli Stack {$[\mathcal{C}^\mu / \mathrm{Ad}_\sigma \mathcal{K}]$} as a Pro-Zip Stack}\label{S:prozip}

In this subsection, we let $k$, $G$, and $\mu$ be as in \S \ref{S:Gzips}. Recall that in \S \ref{S:UnipGps}, we introduced the group subfunctor $\mathcal{E}_{\mu}^{\dagger}$ of $\mathcal{H}^- \times_k \mathcal{H}^+$. Below, we define a similar subfunctor of ${}^-\mathcal{H} \times_k {}^+\mathcal{H}$. 

\begin{definition}\label{Def:ProZipGp} 
	The \emph{pro-zip group} of $\mu$ is defined as the group subfunctor $\mathcal{E}_{\mu} \subseteq {}^-\mathcal{H} \times_k {}^+\mathcal{H}$:
	\begin{align}
		\mathcal{E}_{\mu} = \{(h, g) \in {}^-\mathcal{H} \times_k {}^+\mathcal{H} \;|\; h = \mu(t) g \mu(t)^{-1}\}. 
	\end{align}
\end{definition}

From the definition of $\mathcal{E}_{\mu}$, we obtain the following isomorphisms of $k$-group schemes:
\[
{}^+\mathcal{H} \xrightarrow[\cong]{(\mathrm{Ad}_{\mu(t)}, \mathrm{id})} \mathcal{E}_{\mu} \xrightarrow[\cong]{\mathrm{pr}_1} {}^-\mathcal{H}.
\]

\begin{lemma} The following holds:
	\begin{enumerate}
		\item The map $\mathcal{L}^+ G \times \mathcal{L}^+ G \to \mathcal{C}^{\mu}$ given by $(g_1, g_2) \mapsto g_1^{-1}\mu(t) g_2$ induces an isomorphism of $k$-schemes, 
		\[
		\mathcal{E}_\mu \backslash (\mathcal{L}^+ G \times \mathcal{L}^+ G) \cong \mathcal{C}^{\mu},
		\]
		where $\mathcal{E}_\mu$ acts on $\mathcal{L}^+ G \times_k \mathcal{L}^+ G$ by the rule $(x, y) \cdot (h, g) = (h^{-1} x, g^{-1} y)$ with $(h, g) \in \mathcal{E}_\mu$. 
		
		\item The sequence of $k$-schemes 
		\[
		\mathcal{L}^+ G \xrightarrow{i_2} \mathcal{L}^+ G \times \mathcal{L}^+ G \xrightarrow{\mathrm{pr}} \mathcal{E}_\mu \backslash (\mathcal{L}^+ G \times \mathcal{L}^+ G) \cong \mathcal{C}^{\mu}
		\]
		induces natural isomorphisms of $k$-stacks:
		\[
		[\mathcal{L}^+ G / \mathcal{E}_\mu] \cong [\mathcal{C}^{\mu} / \mathrm{Ad}\mathcal{L}^+ G], \quad [\mathcal{L}^+ G / \mathsf{R}_\sigma \mathcal{E}_\mu] \cong [\mathcal{C}^{\mu} / \mathrm{Ad}_{\sigma}\mathcal{L}^+ G],
		\]
		where $\mathsf{R}_\sigma \mathcal{E}_\mu$ denotes the action of $\mathcal{E}_\mu$ on $\mathcal{L}^+ G$ by the rule $x \cdot (h,g) = g^{-1} x \sigma(h)$ with $x \in \mathcal{L}^+ G$. 
		
		\item The diagrams below are commutative:
		\begin{align}\label{Eq:ZipRed}
			\xymatrix{\mathcal{E}_\mu \backslash (\mathcal{L}^+ G \times \mathcal{L}^+ G) \ar[d] \ar[rr]^{\cong} && \mathcal{C}^{\mu} \ar[d] \\
				\mathsf{E}_{\mu} \backslash (G \times G) \ar[rr]^{\cong} && \prescript{}{1}{\mathcal{C}_1^\mu},} \quad 
			\xymatrix{[\mathcal{L}^+ G / \mathcal{E}_\mu] \ar[d] \ar[r]^{\cong} & [\mathcal{C}^{\mu} / \mathrm{Ad}\mathcal{L}^+ G] \ar[d] \\
				[G / \mathsf{E}_{\mu}] \ar[r]^{\cong} & [\prescript{}{1}{\mathcal{C}_1^\mu} / \mathrm{Ad} G],} \quad  
			\xymatrix{[\mathcal{L}^+ G / \mathsf{R}_\sigma \mathcal{E}_\mu] \ar[d] \ar[r]^{\cong} & [\mathcal{C}^{\mu} / \mathrm{Ad}_{\sigma} \mathcal{L}^+ G] \ar[d] \\
				[G / \mathsf{R}_{\sigma} \mathsf{E}_{\mu}] \ar[r]^{\cong} & [\prescript{}{1}{\mathcal{C}_1^\mu} / \mathrm{Ad}_{\sigma} G].}  
		\end{align}
	\end{enumerate}
\end{lemma}
\begin{proof}
	(1) follows from straightforward verification and (2) follows from an abstract-nonsense argument as in \S \ref{S:ShtukAsZip}.
\end{proof}

The significance of the isomorphism 
$[\mathcal{L}^+ G / \mathsf{R}_\sigma \mathcal{E}_\mu] \cong [\mathcal{C}^{\mu} / \mathrm{Ad}_{\sigma}(\mathcal{L}^+ G)]$
is that it interprets local $G$-shtukas as $\mathcal{L}^+G$-zips. Specifically, for a $k$-scheme $X$, the category of local $G$-shtukas of type~$\mu$ over~$X$ is equivalent to the category of $\mathcal{L}^+G$-zips of type~$\mu$ over~$X$. In addition, the left side of diagram~\eqref{Eq:ZipRed} is integral, so it admits direct reduction modulo~$t^n$. For example, reduction modulo~$t$ of the top of diagram~\eqref{Eq:ZipRed} recovers the bottom isomorphisms.

\section{Application to Zip Period Maps of Shimura Varieties}\label{S:App}

This section logically extends the work presented in \cite{Yan18} and \cite{Yan24}. The setup here follows that of \cite[\S 3]{Yan24}. 

Let $(\textbf{G}, \textbf{X})$ be a Shimura datum of Hodge type, and denote by $\mathcal{S}_{\mathsf{K}}$ the Kisin-Vasiu integral model of the associated Shimura variety $\mathrm{Sh}_{\mathsf{K}}(\textbf{G}, \textbf{X})$ of level $\mathsf{K}=\mathsf{K}_p\mathsf{K}^p$, which is hyperspecial at $p$. This hyperspecial condition on $\mathsf{K}$ implies that $\textbf{G}_{\mathbb{Q}_p}$ admits a reductive $\mathbb{Z}_p$-model $\mathcal{G}$, whose special fibre we denote by $G$, a (connected) reductive group over $\mathbb{F}_p$. Recall that the integral model $\mathcal{S}_{\mathsf{K}}$ is a quasi-projective and smooth scheme over $\mathcal{O}$, the localization at some place above $p$ of the ring of integers of the reflex field of $(\textbf{G}, \textbf{X})$. Write $\kappa$ for the residue field of $\mathcal{O}$, and $S := \mathcal{S}_{\mathsf{K}} \otimes_{\mathcal{O}} \kappa$. Let $\mu: \mathbb{G}_{m,\kappa} \to G_\kappa$ be a representative for the reduction over $\kappa$ of the $ \textbf{G}(\mathbb{C}) $-conjugacy class $[\mu]_{\mathbb{C}}$ containing the inverses of Hodge cocharacters $ \mathbb{G}_{m, \mathbb{C}} \to \textbf{G}_{\mathbb{C}}$ determined by $(\textbf{G}, \textbf{X})$. 

Denote by $P_{\pm} \subseteq G_\kappa$ the opposite parabolic subgroups of $G_\kappa$ defined by $\mu$, and $U_{\pm} \subseteq P_{\pm}$ the corresponding unipotent radicals, and $U_-^\sigma$ as the base change of $U_-$ along $\sigma: \kappa \to \kappa$.

We will apply results from previous sections to the specific situation where $k = \kappa$, $G = G_\kappa$, and $\tau = \sigma$. In particular, we can form the Frobenius zip group $\mathsf{E}_{\mu}^{\mathrm{PWZ}}$ and the zip stack $[G_{\kappa} / \mathsf{E}_{\mu}^{\mathrm{PWZ}}]$. Recall that the zip period map for $S$ is a smooth morphism of algebraic $\kappa$-stacks 
\[
\zeta: S \to [G_{\kappa} / \mathsf{E}_{\mu}^{\mathrm{PWZ}}],
\]
which is constructed in generality (our Shimura variety in question is of Hodge type) by C. Zhang in \cite{ChaoZhangEOStratification}, in order to define and study Ekedahl-Oort strata of $S$. The construction in loc. cit. employs the language of $G$-zips developed in \cite{PinkWedhornZiegler2}. A local construction of $\zeta$, which bypasses the language of $G$-zips and is more adapted to the situation here, is provided in \cite{Yan24}. More precisely, we constructed in loc. cit. a $P_+$-equivariant morphism of smooth $\kappa$-schemes, 
\[
\gamma: \mathrm{I}_+ \to G_{\kappa} / U_-^{\sigma},
\]
inducing the zip period map $\zeta$ mentioned above, where the $\kappa$-scheme $\mathrm{I}_+$ is a $P_+$-torsor over $S$ (see for instance \cite[\S 4.2]{Yan24} for its definition). 

Recall that for the cocharacter $\mu$, apart from the usual Frobenius twist $\sigma(\mu)$, we have also introduced another one, namely $\varphi(\mu)$ in \S \ref{S:TwoFrob}. Recall that they are related by $\varphi(\mu) = p\sigma(\mu)$. Recall also that in \cite{Yan18}, we constructed a morphism of $\kappa$-schemes, $\theta: \mathrm{I}_+ \to \mathcal{C}_1^{\varphi(\mu)}$, and a morphism of fpqc sheaves $\eta: S \to \prescript{}{1}{\mathcal{C}_1^{\varphi(\mu)}}/\mathrm{Ad}_{\sigma} G_{\kappa}$, fitting into the commutative diagram 
\begin{equation*}
	\xymatrix{\mathrm{I}_+ \ar[rr]^{\theta} \ar[d] && \mathcal{C}_1^{\varphi(\mu)} \ar[d]^{\mathrm{Pr}} \\
		S \ar[rr]^{\eta} && \prescript{}{1}{\mathcal{C}_1^{\varphi(\mu)}}/\mathrm{Ad}_{\sigma} G_{\kappa},} 
\end{equation*}
where $\prescript{}{1}{\mathcal{C}_1^{\varphi(\mu)}}/\mathrm{Ad}_{\sigma} G_{\kappa}$ denotes the fpqc quotient sheaf of $\prescript{}{1}{\mathcal{C}_1^{\varphi(\mu)}}$ by $\sigma$-conjugation of $G_{\kappa}$. It was also shown that the geometric fibres of $\eta$ recover the Ekedahl-Oort strata of $S$; cf. \S 2.5 of loc. cit. 

Note that there $\mathcal{C}_1^{\varphi(\mu)}$ is denoted by $\mathcal{D}_1$, and our quotient sheaf  $\prescript{}{1}{\mathcal{C}_1^{\varphi(\mu)}}/\mathrm{Ad}_{\sigma} G_{\kappa}$ here is equal to the sheaf $\mathcal{D}_1/\mathcal{K}^{\diamond}$ in that work.

\begin{theorem} The following holds:
	\begin{enumerate}
		\item The embedding $\mathrm{Emb}_{\sigma(\mu)}: G_{\kappa} \hookrightarrow \mathcal{C}^{\sigma(\mu)}$ induces an immersion of $\kappa$-schemes:
		\[
		\tilde{\alpha}: G_{\kappa}/U_-^{\sigma} \hookrightarrow \mathcal{C}_1^{\sigma(\mu)},
		\]
		and the composition $\alpha: G_{\kappa}/U_-^{\sigma} \to \mathcal{C}_1^{\sigma(\mu)} \to \prescript{}{1}{\mathcal{C}_1^{\sigma(\mu)}}$ is a closed immersion. 
		\item The embedding $\mathrm{Emb}_{\varphi(\mu)}: G_{\kappa} \hookrightarrow \mathcal{C}^{\varphi(\mu)}$ induces an immersion of $\kappa$-schemes:
		\[
		\tilde{\alpha}': G_{\kappa}/U_-^{\sigma} \to \mathcal{C}_1^{\varphi(\mu)},
		\]
		and the composition $\alpha': G_{\kappa}/U_-^{\sigma} \to \mathcal{C}_1^{\varphi(\mu)} \to \prescript{}{1}{\mathcal{C}_1^{\varphi(\mu)}}$ is a closed immersion. 
		\item The association $g \sigma(\mu)(t) h \mapsto g \varphi(\mu)(t) h$ defines an isomorphism of $\kappa$-schemes:
		\[
		\epsilon: \prescript{}{1}{\mathcal{C}_1^{\sigma(\mu)}} \cong \prescript{}{1}{\mathcal{C}_1^{\varphi(\mu)}}, \quad \text{such that} \ \alpha' = \epsilon \circ \alpha.
		\]
		\item The closed immersions $\alpha, \alpha'$ induce morphisms of algebraic $\kappa$-stacks:
		\[
		[\alpha]: G_\kappa \textsf{-Zip}^\mu  \to [\prescript{}{1}{\mathcal{C}_1^{\sigma(\mu)}}/\mathrm{Ad}_{\sigma} G_\kappa], \quad  [\alpha']: G_\kappa \textsf{-Zip}^\mu = G_\kappa \textsf{-Zip}^{p\mu} \to [\prescript{}{1}{\mathcal{C}_1^{\varphi(\mu)}}/\mathrm{Ad}_{\sigma} G_\kappa],
		\]
		inducing homeomorphisms on underlying topological spaces. Moreover, one has $[\alpha']=\epsilon \circ [\alpha]$, where $[\epsilon]:  [\prescript{}{1}{\mathcal{C}_1^{\sigma(\mu)}}/\mathrm{Ad}_{\sigma} G_\kappa] \cong [\prescript{}{1}{\mathcal{C}_1^{\varphi(\mu)}}/\mathrm{Ad}_{\sigma} G_\kappa]$ is the isomorphism induced by $\epsilon$.
		\item The following commutative diagram summarizes the situation, where $\prescript{}{1}{\mathcal{C}_1^{\varphi(\mu)}} / \mathrm{Ad}_{\sigma} G_{\kappa}$ is viewed as the common coarse moduli sheaf of $[\prescript{}{1}{\mathcal{C}_1^{\sigma(\mu)}} / \mathrm{Ad}_{\sigma} G_\kappa]$ and $[\prescript{}{1}{\mathcal{C}_1^{\varphi(\mu)}} / \mathrm{Ad}_{\sigma} G_\kappa]$ obtained by taking isomorphism classes,  
		\begin{equation}\label{Bigdiagram}
			\xymatrix{\mathrm{I}_+ \ar@/^4pc/[rrrr]^{\theta} \ar[dd] \ar[r]^{\gamma} & G_{\kappa}/U_-^{\sigma} \ar[r]_{\tilde{\alpha}} \ar@/^2pc/[rrr]^{\tilde{\alpha}' } \ar[dd] \ar[dr]_{\alpha} & \mathcal{C}_{1}^{\sigma(\mu)} \ar[d] && \mathcal{C}_1^{\varphi(\mu)} \ar[d] \\
				&& \prescript{}{1}{\mathcal{C}_1^{\sigma(\mu)}} \ar[d] \ar[rr]^{\epsilon}_{\cong} && \prescript{}{1}{\mathcal{C}_1^{\varphi(\mu)}} \ar[d] \\
				S \ar[drr]_{\eta} \ar[r]^{\zeta} & G_\kappa \textsf{-Zip}^\mu \ar[r]^{[\alpha] \ \ \ } & [\prescript{}{1}{\mathcal{C}_1^{\sigma(\mu)}}/\mathrm{Ad}_{\sigma} G_\kappa] \ar[rr]^{[\epsilon]}_{\cong} \ar[d] && [\prescript{}{1}{\mathcal{C}_1^{\varphi(\mu)}}/\mathrm{Ad}_{\sigma} G_\kappa] \ar[lld] \\
				&& \prescript{}{1}{\mathcal{C}_1^{\varphi(\mu)}} / \mathrm{Ad}_{\sigma} G_{\kappa} &&.}
		\end{equation}
		\item If we write $\rho: S \to [\prescript{}{1}{\mathcal{C}_1^{\varphi(\mu)}}/\mathrm{Ad}_{\sigma} G_\kappa]$ for the composition morphism as indicated in the diagram above, then the geometric fibres of $\rho$ yield the Ekedahl-Oort strata of $S$.
	\end{enumerate}
\end{theorem}

\begin{proof}
	The assertions (1) and (2) follow from applying Corollary \ref{Cor:Embed} for the cocharacters $\sigma(\mu)$ and $\varphi(\mu)$ here. The assertion (3) follows from applying \eqref{Thm:keyRepMulp} of Theorem \ref{Thm:keyRep} for $\varphi(\mu) = p \sigma(\mu)$. The assertion (4) comes from applying Corollary \ref{Cor:Comp} for the cocharacter $\mu$ and $p \mu$ respectively, noting that $\mathsf{E}_{\mu}^{\mathrm{PWZ}} = \mathsf{E}^{\mathrm{PWZ}}_{p\mu}$. The commutativity of the diagram follows by construction. The last assertion is a consequence of the fact that $\epsilon$ is a homeomorphism and the main result of \cite{ChaoZhangEOStratification}. 
\end{proof}

\begin{remark}
	The commutativity of the top triangle in \eqref{Bigdiagram} implies that, combined with results from \cite{Yan24}, we can reconstruct the morphism $\theta$, which is established in \cite{Yan18} using $G$-adapted deformations of Breuil-Kisin windows. In other words, we see the composition $\tilde{\alpha}' \circ \gamma$ as a new construction of $\theta$. 
	Comparing with \cite{Yan18}, we observe several improvements regarding the maps $\theta$ and $\eta$: 
	\begin{enumerate}
		\item The morphism of \emph{sheaves} $\eta: S \to \prescript{}{1}{\mathcal{C}_1^{\varphi(\mu)}}/\mathrm{Ad}_{\sigma} G_\kappa$ constructed in loc. cit. is updated to a morphism of \emph{algebraic stacks}, $\rho: S \to [\prescript{}{1}{\mathcal{C}_1^{\varphi(\mu)}}/\mathrm{Ad}_{\sigma} G_\kappa]$.
		\item Up to language, it is essentially shown in loc. cit. that the morphism $[\alpha]$ induces a \emph{bijection} between underlying topological spaces. It is now clear that this bijection is, in fact, a \emph{homeomorphism.}
		\item Although verifying that $\theta = \tilde{\alpha}' \circ \gamma$ is trivial, our new approach of structuring $\theta$ as the composition $\tilde{\alpha} \circ \gamma$ clarifies its structure significantly; for instance, it was not clear from loc. cit. that the map $\theta$ factors through $\gamma$, and likewise, it was not obvious from loc. cit. that the map $\eta$ factors through the zip period map $\zeta$.
	\end{enumerate}
	
	\begin{remark}
		We view the morphism $\rho: S \to [\prescript{}{1}{\mathcal{C}_1^{\varphi(\mu)}}/\mathrm{Ad}_{\sigma} G_\kappa]$ as a variant of the zip period map for $S$. We find $\rho$ and its covering \[\tilde{\rho}:= \epsilon \circ \alpha \circ \gamma: \mathrm{I}_+ \to \prescript{}{1}{\mathcal{C}_1^{\varphi(\mu)}},\] conceptually interesting, as it connects the number field side with the function field side: it is a morphism from the mod $p$ reduction of a Shimura variety to the moduli stack of 1-1 truncated local $G_\kappa$-shtukas of the same type up to Frobenius twist, whose geometric fibres yield the Ekedahl-Oort strata of $S$. In fact, we show in \cite{yan25intperdmap} that the map $\zeta: S \to G_\kappa \textsf{-Zip}^\mu$ factors though the moduli stack $[\prescript{}{1}{\mathcal{C}_1^{\mu}}/\mathrm{Ad}_{\sigma} G_\kappa]$ in prismatic topology of $S$. 
	\end{remark}

\end{remark}

\end{document}